\documentclass[10pt]{amsart}

\usepackage{enumerate}
\usepackage{amsmath,amssymb}
\usepackage{color}
\usepackage{mathrsfs}

\newtheorem{theorem}{Theorem}
\newtheorem{lemma}[theorem]{Lemma}
\newtheorem{proposition}[theorem]{Proposition}
\newtheorem{corollary}[theorem]{Corollary}

\theoremstyle{definition}
\newtheorem{example}[theorem]{Example}
\newtheorem{question}[theorem]{Question}

\begin{document}
\title[Isometries between C$^*$-algebras]{Isometries between C$^*$-algebras\\ with finite corank}

\author[M. Mori]{Michiya Mori}

\dedicatory{Dedicated to Professor Fumio Hiai on the occasion of his $80$th birthday.}

\address{Faculty of Engineering, Niigata University, 8050 Ikarashi 2-no-cho, Nishi-ku, Niigata 950-2181, Japan; Center for Interdisciplinary Theoretical and Mathematical Sciences (iTHEMS), RIKEN, 2-1 Hirosawa, Wako, Saitama 351-0198, Japan.}
\email{michiya.mori.eng@niigata-u.ac.jp}

\thanks{The author was supported by JSPS KAKENHI Grant Number 22K13934.}


\date{}

\begin{abstract}
We give a characterization of linear isometries with finite corank between two arbitrary unital C$^*$-algebras in terms of Jordan $^*$-homomorphisms and ``finite-dimensional remainders''. Also, the corresponding results are given for linear isometries and linear order embeddings between self-adjoint parts of C$^*$-algebras. In the finite-dimensional case, we give new examples of linear isometries between matrix algebras. 
\end{abstract}

\maketitle

\section{Introduction}
Throughout this paper, we assume that $\mathcal{A}, \mathcal{B}$ are general unital C$^*$-algebras. The symbols $\mathcal{U}(\mathcal{A})$, $\mathcal{P}(\mathcal{A})$, and $\mathcal{A}_{sa}$ denote the unitary group, the set of projections, and the self-adjoint part of $\mathcal{A}$, respectively.

In this paper, we are interested in linear isometries $\Phi\colon \mathcal{A}\to\mathcal{B}$, that is, linear mappings satisfying $\lVert \Phi(a_1)-\Phi(a_2)\rVert = \lVert a_1-a_2\rVert$ for any pair $a_1, a_2\in \mathcal{A}$, or equivalently, linear mappings satisfying $\lVert \Phi(a)\rVert = \lVert a\rVert$ for any $a\in \mathcal{A}$. 
In 1951, Kadison characterized surjective linear isometries in terms of Jordan $^*$-isomorphisms. 
Recall that a linear mapping (resp. a linear bijection) $J\colon \mathcal{A}\to \mathcal{B}$ is called a Jordan $^*$-homomorphism (resp. a Jordan $^*$-isomorphism) if it satisfies $J(a^*)=J(a)^*$ and $J(a_1a_2+a_2a_1)=J(a_1)J(a_2)+J(a_2)J(a_1)$ for any $a, a_1, a_2\in \mathcal{A}$. 
\begin{theorem}[Kadison {\cite[Theorem 7]{K}}]\label{kadison51}
 If $\Phi\colon \mathcal{A}\to\mathcal{B}$ is a surjective linear isometry, then $\Phi(1)\in \mathcal{B}$ is unitary and the mapping $\mathcal{A}\ni a\mapsto \Phi(1)^*\Phi(a)\in \mathcal{B}$ is a Jordan $^*$-isomorphism. 
\end{theorem}

What happens if we drop the assumption of surjectivity?
Some local properties are known for nonsurjective linear isometries between C$^*$-algebras or related spaces, see \cite{CW}, \cite{CM}, \cite{AP}.
However, in general, it appears to be very difficult to study the global behavior of linear isometries without the surjectivity assumption. 

\subsection{Linear isometries between matrix algebras}\label{subsec:matrix}
Indeed, no complete description of linear isometries is known even if we restrict our attention to the case of matrix algebras. 
In the paper \cite{CLP}, Cheung, Li, and Poon gave partial results and several problems on such mappings. 
For $n\geq 0$, let $\mathbb{M}_n$ denote the space of $n\times n$ matrices with complex entries.
Let $\alpha, \beta$ be nonnegative integers satisfying $0<\alpha+\beta$, and let $n,m$ be positive integers.
Let $T_1, T_2$ be $n(\alpha+\beta)\times m$ matrices satisfying $\lVert T_1\rVert, \lVert T_2\rVert \leq 1$ and $\mathrm{rank}\, (I_{n(\alpha+\beta)}-T_1T_1^*)+\mathrm{rank}\,(I_{n(\alpha+\beta)}-T_2T_2^*)<\alpha+\beta$. Then one may show via basic matrix analysis that the mapping $\mathbb{M}_n\ni X\mapsto T_1^*((X\otimes I_\alpha)\oplus (X^t\otimes I_\beta))T_2\in \mathbb{M}_{m}$ is a linear isometry. 
Therefore, for any integer $k\geq m$, any linear mapping $f\colon \mathbb{M}_n\to \mathbb{M}_{k-m}$ of norm at most $1$, and for any pair of unitaries $V_1, V_2\in \mathbb{M}_k$, we see that the mapping 
\begin{equation}\label{clp}
\Phi\colon \mathbb{M}_n\ni X\mapsto V_1(T_1^*((X\otimes I_\alpha)\oplus (X^t\otimes I_\beta))T_2\oplus f(X))V_2^* \in \mathbb{M}_k
\end{equation}
is also an isometry \cite[Proposition 4.1]{CLP}. 
In this paper, we call a mapping $\Phi \colon \mathbb{M}_n\to \mathbb{M}_k$ a \emph{CLP isometry} if it is of the above form for some suitable choice of $\alpha, \beta, m$, $f$, and $T_1, T_2$, $V_1, V_2$ as above.
Cheung, Li, and Poon asked whether every linear isometry $\Phi \colon \mathbb{M}_n\to \mathbb{M}_k$ is a CLP isometry \cite[p.\ 15, (1)]{CLP}, and proved that this is the case if $k\leq 2n-1$ \cite[Theorem 1.1]{CLP}.

In the proof of this result, it was essential to think of mappings on the space of hermitian matrices. 
Let $\mathscr{H}_n$ denote the space of $n\times n$ hermitian matrices.
If $\alpha, \beta, \delta$ are nonnegative integers satisfying $\delta<\alpha+\beta$, and if $U\colon \mathbb{C}^{n(\alpha+\beta)-\delta} \to \mathbb{C}^{n(\alpha+\beta)}$ is a linear isometry, then 
\begin{equation}\label{phi0}
\Phi _0\colon \mathscr{H}_n\ni X\mapsto U^*((X\otimes I_\alpha)\oplus (X^t\otimes I_\beta))U\in \mathscr{H}_{n(\alpha+\beta)-\delta}
\end{equation}
is a unital positive linear isometry.  
Therefore, for any integer $k\geq n(\alpha+\beta)-\delta$, any unital positive linear mapping $f\colon \mathscr{H}_n\to \mathscr{H}_{k-n(\alpha+\beta)+\delta}$ (here, $f$ can be the zero mapping in the case $k= n(\alpha+\beta)-\delta$), and for any unitary $V\in \mathbb{M}_k$, we see that the mapping 
\begin{equation}\label{clp2}
\mathscr{H}_n\ni X\mapsto V(\Phi_0(X)\oplus f(X))V^* \in \mathscr{H}_k
\end{equation}
is also a unital positive linear isometry \cite[Proposition 4.2]{CLP}. 
Let us again use the terminology \emph{CLP isometry} for a mapping $\Phi \colon \mathscr{H}_n\to \mathscr{H}_k$ that can be obtained in this manner for some suitable choice of $\alpha, \beta, \delta$, $f$, and $U, V$ as above. (It is unlikely that the use of the same terminology causes confusions because different spaces are considered.)
Cheung, Li, and Poon asked whether every unital positive linear isometry $\Phi \colon \mathscr{H}_n\to \mathscr{H}_k$ is a CLP-isometry \cite[p.\ 15, (2)]{CLP} (see also \cite[Theorem 2.3]{CLP}), and proved that this is the case if $k\leq 2n-2$ \cite[Theorem 1.2]{CLP}. 
They also asked the same problem under the additional assumption of complete positivity or decomposability \cite[p.\ 15, (4)]{CLP}.

In Section \ref{secexample}, we solve all these problems negatively by exhibiting examples.
Consequently, at present there is no clear understanding on the general form of linear isometries between matrix algebras.

\subsection{Linear isometries between commutative C$^*$-algebras}
If we restrict our attention to linear isometries between commutative C$^*$-algebras, one may find a lot of works in the literature.
For a compact Hausdorff space $X$, let $C(X)$ denote the commutative C$^*$-algebra of continuous complex-valued functions on $X$.
In what follows, let $X, Y$ be compact Hausdorff spaces. 
It is well-known as the Banach--Stone theorem from 1930s that a surjective linear isometry $\Phi\colon C(X)\to C(Y)$ is of the form $\Phi(f)(y)=\tau(y)f(\eta(y))$ for every $f\in C(X)$ and every $y\in Y$, where $\tau\colon Y\to \mathbb{T}:= \{z\in \mathbb{C}\mid \lvert z\rvert=1\}$ is a continuous function and $\eta\colon Y\to X$ is a homeomorphism. 
In 1966, Holszty\'nski \cite[Theorem]{Hols} proved that if $\Phi \colon C(X)\to C(Y)$ is a linear isometry then one can take a certain nonempty closed subset $Y_0\subset Y$, a continuous mapping $\tau\colon Y_0\to \mathbb{T}$, and a continuous surjective mapping $\eta\colon Y_0\to X$ satisfying $\Phi(f)(y)=\tau(y)f(\eta(y))$ for every $f\in C(X)$ and every $y\in Y_0$.

Holszty\'nski's theorem is generalized and modified in various directions. For example, Cambern \cite{C} gave a result on isometries between vector-valued continuous functions. Jeang and Wong \cite{JW0} gave a version of Holszty\'nski's result for commutative C$^*$-algebras that are not necessarily unital.
Araujo and Font \cite{AF0} studied linear isometries between subspaces of commutative C$^*$-algebras under certain assumptions.
Jim\'enez-Vargas and Villegas-Vallecillos \cite{JV} worked on linear isometries between the spaces of Lipschitz functions.

On the other hand, Holszty\'nski's theorem only gives a partial information on linear isometries. Therefore, we need some additional assumption to get a better description of such mappings. One direction of research which flourished in the past 40 years is the study of linear isometries of finite corank, and in particular, of corank $1$.

\subsection{Linear isometries of finite corank}
Recall that a linear mapping $\Phi\colon \mathcal{V}\to \mathcal{W}$ between vector spaces is said to have \emph{corank $n$} if the quotient vector space $\mathcal{W}/\Phi(\mathcal{V})$ is $n$-dimensional, and it is said to have \emph{finite corank} if it has corank $n$ for some integer $n\geq 0$. 
The best-known example of an isometry of corank $1$ is probably the \emph{unilateral shift}, that is, the linear isometry on a Hilbert space with orthonormal basis $\phi_n$, $n\geq 1$, that sends each $\phi_n$ to $\phi_{n+1}$. 
Recall that the famous Wold--von Neumann theorem assures that the unilateral shift is a building block of general linear isometries on a Hilbert space and it verifies the importance of the unilateral shift.

With the motivation to get a generalization of the unilateral shift to the setting of Banach spaces, Crownover \cite{Cr} introduced the following concept. 
A linear mapping $\Phi\colon \mathcal{X}\to \mathcal{X}$ on a Banach space $\mathcal{X}$ is called an \emph{isometric shift} if it is an isometry with corank $1$ satisfying $\bigcap_{n\geq 1} \Phi^n(\mathcal{X})=\{0\}$. Crownover \cite[Theorem 1]{Cr} proved that a Banach space with an isometric shift can be identified isometrically with a certain space of sequences.

Because commutative C$^*$-algebras form an important class of Banach spaces, it is natural to think about isometric shift, or more generally, isometries with finite corank, for commutative C$^*$-algebras.
In 1988, Holub \cite{Holu} studied the problem of when the real Banach space $C(X)_{sa}$ of real-valued continuous functions admits an isometric shift. In 1991, using Holszty\'nski's theorem \cite[Theorem]{Hols}, Gutek, Hart, Jamison, and Rajagopalan \cite{GHJR} gave the following result on general isometries of corank $1$ on $C(X)$, and gave answers to several questions posed by Holub. 
\begin{theorem}[Gutek, Hart, Jamison, and Rajagopalan {\cite[Theorem 2.1]{GHJR}}]\label{th:ghjr}
Let $\Phi\colon C(X)\to C(X)$ be a linear isometry with corank $1$. 
Then there are an open subset $X_1\subset X$ with at most one point, a continuous mapping $\tau\colon X\setminus X_1\to \mathbb{T}$, and a continuous surjective mapping $\eta\colon X\setminus X_1\to X$ satisfying $\Phi(f)(x)=\tau(x)f(\eta(x))$ for every $f\in C(X)$ and every $x\in X\setminus X_1$. %The same statement holds when $C(X)$ is replaced with $C(X)_{sa}$.
\end{theorem}
(In the statement of \cite[Theorem 2.1]{GHJR}, the mapping is assumed to be an isometric shift, but note that the assumption $\bigcap_{n\geq 1} \Phi^n(C(X))=\{0\}$ is not used in the proof.)
The mapping $\Phi$ as in Theorem \ref{th:ghjr} is said to be of \emph{type I} if one may take $X_1$ with $\# X_1=1$, and of \emph{type II} if one may take the empty set as $X_1$. Note that an isometric shift $\Phi\colon C(X)\to C(X)$ can be of both types I and II, see \cite[Example 3]{GHJR}.

After the work of Gutek, Hart, Jamison, and Rajagopalan, many researchers worked on isometric shifts on $C(X)$, often with the view on the dynamical properties of iterates of $\Phi$. See for example \cite{FV}, \cite{H}, \cite{AF2}, \cite{GN}, \cite{A}, and \cite{A2}. 
Isometries of finite corank (and in particular, of corank $1$) have been studied for several other spaces of functions as well.
For example, Jeang and Wong \cite{JW} studied isometries of finite corank between commutative C$^*$-algebras that are not necessarily unital. 
Izuchi \cite{I} studied isometries of corank $1$ on Douglas algebras. 
Koshimizu \cite{Ko} studied isometries of finite corank on the spaces of differentiable functions and Lipschitz functions. 
Araujo and Font \cite{AF}, \cite{AF3} studied the case of function algebras and that of the spaces of vector-valued continuous functions.

\subsection{Main results of this paper}
As reviewed above, the literature on isometries of finite corank between commutative C$^*$-algebras is rich. 
However, it seems that there is almost no such work for noncommutative C$^*$-algebras in the existing literature. 
The main purpose of this paper is to fill in this gap. 
Namely, we study linear isometries of finite corank between general C$^*$-algebras.

Recall that $\mathcal{A}$, $\mathcal{B}$ are unital C$^*$-algebras.
The following is one of the main results of this paper. This gives a variant of Kadison's theorem (Theorem \ref{kadison51}).

\begin{theorem}\label{main}
Let $\Phi \colon \mathcal{A}\to \mathcal{B}$ be a linear isometry with finite corank.
Then there exist central projections $p\in \mathcal{A}$, $q\in \mathcal{B}$ (some of which can be $0$ or $1$) such that $p\mathcal{A}$, $q\mathcal{B}$ are finite-dimensional, and there exist
\begin{itemize}
\item a linear isometry $\Phi _1\colon p\mathcal{A}\to q\mathcal{B}$, 
\item a linear isometry $\Phi _2\colon (1-p)\mathcal{A}\to (1-q)\mathcal{B}$ that is the composition $\Phi _2=vJ(\cdot)$ of an injective unital Jordan $^*$-homomorphism $J\colon (1-p)\mathcal{A}\to (1-q)\mathcal{B}$ with finite corank and the left-multiplication by a unitary $v=(1-q)\Phi(1)\in \mathcal{U}((1-q)\mathcal{B})$, and
\item a linear mapping $\Phi _3\colon (1-p)\mathcal{A}\to q\mathcal{B}$, 
\end{itemize}
satisfying the following two conditions: 
\begin{itemize}
\item For every $a\in \mathcal{A}$, we have $\Phi (a)=\Phi_1(pa)+\Phi_2((1-p)a)+\Phi_3((1-p)a)$, and 
\item the operator norm of the mapping $\mathcal{A}\ni a\mapsto \Phi_1(pa)+\Phi_3((1-p)a)\in q\mathcal{B}$ is at most $1$.
\end{itemize}
\end{theorem}
Therefore, such an isometry has a rather special form.
Observe that the equation 
\begin{equation}\label{corank}
\mathrm{corank}\, \Phi =\mathrm{corank}\, \Phi_1+\mathrm{corank}\, \Phi_2=\mathrm{corank}\, \Phi_1+\mathrm{corank}\, J
\end{equation}
holds. Indeed, the projection from $\mathcal{B}$ onto $(1-q)\mathcal{B}$ induces a surjection from $\mathcal{B}/\Phi(A)$ onto $(1-q)\mathcal{B}/\Phi_2((1-p)\mathcal{A})$. 
On the other hand, the injectivity of $\Phi_2$ implies that the inclusion $q\mathcal{B}\subset \mathcal{B}$ induces an injection from  $q\mathcal{B}/\Phi_1(p\mathcal{A})$ to $\mathcal{B}/\Phi(A)$. 
Since the kernel of the former mapping equals the range of the latter, we get the equation.

Let us show that the converse of this theorem holds. 
Let $\Phi$ be as in the conclusion of this theorem. Then it is clear that $\Phi$ has finite corank.
We show that $\Phi$ is an isometry.
Let $a\in \mathcal{A}$. Then the norm of $\Phi(a)=\Phi_1(pa)+\Phi_2((1-p)a)+\Phi_3((1-p)a)$ equals 
$\max\{\lVert \Phi_1(pa)+\Phi_3((1-p)a)\rVert, \lVert\Phi_2((1-p)a)\rVert\}$. 
We have  $\lVert \Phi_1(pa)+\Phi_3((1-p)a)\rVert\leq \lVert a\rVert$ and $\lVert\Phi_2((1-p)a)\rVert=\lVert(1-p)a\rVert$.  
Therefore, if $\lVert a\rVert = \lVert(1-p)a\rVert$, then we obtain $\lVert \Phi(a)\rVert =\lVert a\rVert$. 
If $\lVert a\rVert = \lVert pa\rVert$, then 
$\lVert \Phi_1(pa)\pm \Phi_3((1-p)a)\rVert\leq \lVert pa \pm (1-p)a\rVert = \lVert pa\rVert$ holds. Hence we have 
\[
2\lVert pa\rVert=2\lVert \Phi_1(pa)\rVert \leq \lVert \Phi_1(pa)+ \Phi_3((1-p)a)\rVert+\lVert \Phi_1(pa)- \Phi_3((1-p)a)\rVert\leq 2\lVert pa\rVert
\]
and $\lVert \Phi_1(pa)\pm \Phi_3((1-p)a)\rVert=\lVert pa\rVert$, so we get $\lVert \Phi(a)\rVert =\lVert a\rVert$.
Consequently, this theorem gives a characterization of linear isometries of unital C$^*$-algebras with finite corank. 

We also have a self-adjoint version of the above theorem, which gives a variant of another theorem by Kadison \cite[Theorem 2]{K2} in 1952.

\begin{theorem}\label{mainsa}
Let $\Phi \colon \mathcal{A}_{sa}\to \mathcal{B}_{sa}$ be a linear isometry with finite corank.
Then there exist central projections $p\in \mathcal{A}$, $q\in \mathcal{B}$ (some of which can be $0$ or $1$) such that $p\mathcal{A}$, $q\mathcal{B}$ are finite-dimensional, and there exist
\begin{itemize}
\item a linear isometry $\Phi _1\colon p\mathcal{A}_{sa}\to q\mathcal{B}_{sa}$, 
\item a linear isometry $\Phi _2\colon (1-p)\mathcal{A}_{sa}\to (1-q)\mathcal{B}_{sa}$ that satisfies $\Phi _2(a)=vJ(a)$ for every $a\in (1-p)\mathcal{A}_{sa}$, where $J\colon (1-p)\mathcal{A}\to (1-q)\mathcal{B}$ is an injective unital Jordan $^*$-homomorphism with finite corank and $v=(1-q)\Phi(1)$ is a central self-adjoint unitary in $(1-q)\mathcal{B}$, and 
\item a linear mapping $\Phi _3\colon (1-p)\mathcal{A}_{sa}\to q\mathcal{B}_{sa}$, 
\end{itemize}
satisfying the following two conditions: 
\begin{itemize}
\item For every $a\in \mathcal{A}_{sa}$, we have $\Phi (a)=\Phi_1(pa)+\Phi_2((1-p)a)+\Phi_3((1-p)a)$, and 
\item the operator norm of the mapping $\mathcal{A}_{sa}\ni a\mapsto \Phi_1(pa)+\Phi_3((1-p)a)\in q\mathcal{B}_{sa}$ is at most $1$.
\end{itemize}
\end{theorem}
The conclusion of the latter theorem is quite similar to the former, but notice that $v$ needs to be a self-adjoint unitary in the latter theorem. 

Even though Theorems \ref{main} and \ref{mainsa} give characterizations of linear isometries between (self-adjoint parts of) C$^*$-algebras with finite corank, we still do not have a good understanding of the mappings $\Phi _1, \Phi_2, \Phi_3$ in the theorems. 
In particular, $\Phi _1\colon p\mathcal{A}\to q\mathcal{B}$ in Theorem \ref{main} is a linear isometry of finite-dimensional C$^*$-algebras. However, as mentioned before, at present there is little understanding on such a mapping even in the special case of matrix algebras.

\subsection{Remarks on the main theorem}
In this paper, a linear subspace $\mathcal{A}_1\subset \mathcal{A}$ is called a \emph{Jordan $^*$-subalgebra} if $a^*$, $a_1a_2+a_2a_1\in \mathcal{A}_1$ for any $a, a_1, a_2\in \mathcal{A}_1$.
Let us think about the mapping $J$ in Theorem \ref{main} or \ref{mainsa}. 
To examine it, it is essential to study finite-codimensional closed Jordan $^*$-subalgebras of C$^*$-algebras. 
This can be understood by the following proposition.
\begin{proposition}\label{J1}
Let $\mathcal{A}_1\subset \mathcal{A}$ be a closed Jordan $^*$-subalgebra with codimension $m$.
Then there is a surjective $^*$-homomorphism $\pi\colon \mathcal{A}\to \mathcal{A}_2$ onto a finite-dimensional C$^*$-algebra $\mathcal{A}_2$ and a Jordan $^*$-subalgebra $\mathcal{A}_3\subset \mathcal{A}_2$ such that $m=\dim \mathcal{A}_2-\dim \mathcal{A}_3$ and $\mathcal{A}_1=\{a\in \mathcal{A}\mid \pi(a)\in \mathcal{A}_3\}$.
\end{proposition}
Therefore, the investigation on finite-codimensional closed Jordan $^*$-subalgebras reduces to that of the finite-dimensional $^*$-representation $\pi\colon \mathcal{A}\to \mathcal{A}_2$ and the inclusion $\mathcal{A}_3\subset \mathcal{A}_2$.

Now, we would like to consider the special case of corank $1$ in the main theorem.
To study this case, it is helpful to think of the case of $m=1$ in Proposition \ref{J1}.

\begin{proposition}\label{J2}
Let $\mathcal{A}_1\subset \mathcal{A}$ be a closed Jordan $^*$-subalgebra with codimension $1$.
Then one (and only one) of the following holds.
\begin{itemize}
\item There is a character (i.e. a $^*$-representation with $1$-dimensional range) $\varphi\colon \mathcal{A}\to \mathbb{C}$ such that $\mathcal{A}_1=\ker \varphi$. 
\item There are distinct characters $\varphi_1, \varphi_2\colon \mathcal{A}\to \mathbb{C}$ such that $\mathcal{A}_1=\{a\in \mathcal{A}\mid \varphi_1(a)=\varphi_2(a)\}$. 
\item There is a surjective $^*$-representation $\pi\colon \mathcal{A}\to \mathbb{M}_2$ such that $\mathcal{A}_1=\{a\in \mathcal{A}\mid \pi(a)_{11}=\pi(a)_{22}\}$. (Here, for a matrix $X\in \mathbb{M}_2$, $X_{ij}$ denotes its $(i,j)$-entry.)
\end{itemize}
\end{proposition}
Note that $\mathcal{A}_1\subset \mathcal{A}$ is a C$^*$-subalgebra in the first two cases, but the same doesn't hold in the third case.
This proposition together with the following result gives a more definitive description of isometries with corank $1$.

\begin{proposition}\label{corank1}
Let $\Phi\colon\mathcal{A}\to \mathcal{B}$ be a linear isometry with corank $1$.
Then one of the following holds.
\begin{itemize}
\item There is a central projection $q_0$ in $\mathcal{B}$ such that $\dim q_0\mathcal{B}=1$, and the mapping $\mathcal{A}\ni a\mapsto (1-q_0)\Phi(a)\in (1-q_0)\mathcal{B}$ is a linear surjective isometry (hence we may use Theorem \ref{kadison51} to this part). In this case, the mapping $\mathcal{A}\ni a\mapsto q_0\Phi(a)\in q_0\mathcal{B}$ is of norm at most $1$.
\item The operator $\Phi(1)\in \mathcal{B}$ is unitary, and the mapping $\mathcal{A}\ni a\mapsto \Phi(1)^*\Phi(a)\in \mathcal{B}$ is an injective unital Jordan $^*$-homomorphism with corank $1$. (In this case, we may apply Proposition \ref{J2} to the image of this Jordan $^*$-homomorphism.)
\end{itemize}
\end{proposition}
By this, for a linear isometry of corank $1$, the difficulty in thinking about ``finite-dimensional remainder'' essentially disappears. Let us verify that this proposition actually generalizes Theorem \ref{th:ghjr}.
Assume that $\mathcal{A}=\mathcal{B}=C(X)$ is a commutative C$^*$-algebra.
If the first item holds, then there is an isolated point $x_0\in X$ such that $q_0$ equals the indicator function for $\{x_0\}\subset X$. We get a surjective isometry $C(X)\ni a\mapsto (1-q_0)\Phi(a)\in (1-q_0)C(X)=C(X\setminus\{x_0\})$ between commutative C$^*$- algebras. It then follows from the Banach--Stone theorem that this mapping can be written using a homeomorphism from $X\setminus\{x_0\}$ onto $X$, which leads to the fact that $\Phi$ is of type I.
If the second item holds, then the mapping $\mathcal{A}\ni a\mapsto \Phi(1)^*\Phi(a)\in \mathcal{B}$ is an injective $^*$-homomorphism between commutative C$^*$-algebras with corank $1$. It then follows that $\Phi$ is of type II.
We also give a self-adjoint variant of Proposition \ref{corank1}.
\begin{proposition}\label{corank1sa}
Let $\Phi\colon\mathcal{A}_{sa}\to \mathcal{B}_{sa}$ be a linear isometry with corank $1$.
Then one of the following holds.
\begin{itemize}
\item There is a central projection $q_0$ in $\mathcal{B}$ such that $\dim q_0\mathcal{B}=1$, and the mapping $\mathcal{A}_{sa}\ni a\mapsto (1-q_0)\Phi(a)\in (1-q_0)\mathcal{B}_{sa}$ is a linear surjective isometry. In this case, the mapping $\mathcal{A}_{sa}\ni a\mapsto q_0\Phi(a)\in q_0\mathcal{B}_{sa}$ is of norm at most $1$.
\item The operator $\Phi(1)\in \mathcal{B}$ is a central self-adjoint unitary, and the mapping $\mathcal{A}_{sa}\ni a\mapsto \Phi(1)\Phi(a)\in \mathcal{B}$ extends to an injective unital Jordan $^*$-homomorphism from $\mathcal{A}$ to $\mathcal{B}$ with corank $1$. 
\end{itemize}
\end{proposition}

We next consider another special case where $\mathcal{A}$ does not have a finite-dimensional ideal. 
\begin{corollary}
Assume that $\dim \mathcal{A}=\infty$ and $\mathcal{A}$ has no finite-dimensional ideal except for $\{0\}$. 
If $\Phi\colon \mathcal{A}\to \mathcal{B}$ is a linear isometry of finite corank, then there exists a central projection $q\in \mathcal{B}$ (which can be $0$) such that $q\mathcal{B}$ is finite-dimensional, and there exist
\begin{itemize}
\item a linear isometry $\Phi _2\colon \mathcal{A}\to (1-q)\mathcal{B}$ that is the composition $\Phi _2=vJ(\cdot)$ of an injective unital Jordan $^*$-homomorphism $J\colon \mathcal{A}\to (1-q)\mathcal{B}$ with finite corank and the left-multiplication by a unitary $v\in \mathcal{U}((1-q)\mathcal{B})$, and
\item a linear mapping $\Phi _3\colon \mathcal{A}\to q\mathcal{B}$ with norm at most $1$, 
\end{itemize}
satisfying $\Phi (a)=\Phi_2(a)+\Phi_3(a)$ for every $a\in \mathcal{A}$.
\end{corollary}

If we assume instead that $\mathcal{B}$ has such a property, then we get a result that resembles Kadison's theorem (Theorem \ref{kadison51}).
\begin{corollary}\label{hamada}
Assume that $\dim \mathcal{B}=\infty$ and $\mathcal{B}$ has no finite-dimensional ideal except for $\{0\}$. 
If $\Phi\colon \mathcal{A}\to \mathcal{B}$ is a linear isometry of finite corank, then $\Phi(1)\in \mathcal{B}$ is unitary and the mapping $\mathcal{A}\ni a\mapsto \Phi(1)^*\Phi(a)\in \mathcal{B}$ is an injective unital Jordan $^*$-homomorphism with finite corank. 
\end{corollary}
The above two corollaries are immediate from Theorem \ref{main}.
It is needless to say that there are self-adjoint variants of the above two corollaries. We omit them.

\subsection{Order embeddings}
A mapping $\Psi\colon X_1\to X_2$ between two partially ordered sets $(X_1, \leq_1)$, $(X_2, \leq_2)$ is called an \emph{order isomorphism} (resp. an \emph{order embedding}) if it is a bijection (resp. a mapping) satisfying $x\leq_1 y \iff \Psi(x)\leq_2\Psi(y)$ for every pair $x,y\in X_1$. 
In the paper \cite{K2}, Kadison also gave a characterization of linear order isomorphisms from $\mathcal{A}_{sa}$ onto $\mathcal{B}_{sa}$. 
\begin{theorem}[Kadison {\cite[Corollary 5]{K2}}]\label{kadison52}
If $\Psi\colon \mathcal{A}_{sa}\to\mathcal{B}_{sa}$ is a linear order isomorphism with $\Psi(1)=1$, then $\Psi$ extends to a Jordan $^*$-isomorphism from $\mathcal{A}$ onto $\mathcal{B}$. 
\end{theorem}
The following theorem combined with Theorem \ref{mainsa} gives a variant of Theorem \ref{kadison52}. 
\begin{theorem}\label{embedding}
A mapping $\Phi\colon \mathcal{A}_{sa}\to\mathcal{B}_{sa}$ is a linear order embedding with finite corank if and only if there are
\begin{itemize} 
\item a projection $f\in \mathcal{B}$ with $\dim f\mathcal{B}<\infty$, 
\item a positive invertible operator $b\in (1-f)\mathcal{B}(1-f)$, and
\item a unital linear isometry $\Psi \colon \mathcal{A}_{sa}\to (1-f)\mathcal{B}_{sa}(1-f)$ with finite corank
\end{itemize} 
satisfying $\Phi(a) = b\Psi(a)b$ for every $a\in \mathcal{A}_{sa}$.
\end{theorem}

\subsection{Structure of this paper}
This paper is organized as follows. 
In the next section, we give examples of linear isometries between matrix algebras. 
Section \ref{preliminaries} is for preliminaries. We give a result on the extreme points of the closed unit ball for a finite-codimensional subspace of a C$^*$-algebra. We also introduce the concept of a petty subset of a C$^*$-algebra, which will be repeatedly used in the subsequent sections.
Theorem \ref{main} (resp. Theorem \ref{mainsa}) together with Proposition \ref{corank1} (resp. Proposition \ref{corank1sa}) is proved in Section \ref{secmain}  (resp. Section \ref{S3}).
Proofs of other results are collected in Section \ref{other}.
In the final section, we pose further research directions.

\section{Examples of isometries between matrix algebras}\label{secexample}
The results in this section are basically independent from those in the subsequent sections. 
Thus, those readers who are uninterested can skip this section. 
The examples below arise by generalizing the definition of CLP isometries. 
For simplicity, we mainly think of mappings on the space of hermitian matrices. 

Let $n$, $\alpha, \beta, \zeta\geq 0$ be integers, and let $f\colon \mathscr{H}_n\to \mathscr{H}_\zeta$ be a unital positive linear mapping. 
Assume that $\mathcal{K}\subset \mathbb{C}^{n(\alpha+\beta)+\zeta}$ is a subspace such that 
\[
\mathcal{K}\cap \{(h\otimes h_1) \oplus (\overline{h}\otimes h_2)\oplus 0_{\mathbb{C}^\zeta}\mid h_1\in \mathbb{C}^\alpha, h_2\in \mathbb{C}^\beta \}\neq \{0\} 
\] 
for every nonzero vector $h\in \mathbb{C}^n$.
Set $k:=\dim \mathcal{K}$, and take a linear surjective isometry $V\colon \mathbb{C}^k\to \mathcal{K}$. We view $V$ as a linear isometry from $\mathbb{C}^k$ to $\mathbb{C}^{n(\alpha+\beta)+\zeta}$.
Note that the operator $VV^*$ is the projection onto $\mathcal{K}$.
Consider the mapping 
\[
\Phi\colon \mathscr{H}_n\ni X\mapsto V^*((X\otimes I_\alpha)\oplus (X^t\otimes I_\beta)\oplus f(X))V\in \mathscr{H}_k.
\]
Let us call a mapping that can be obtained in this manner a \emph{generalized CLP isometry}.
Since the dimension of the subspace $\{(h\otimes h_1) \oplus (\overline{h}\otimes h_2)\oplus 0_{\mathbb{C}^\zeta}\mid h_1\in \mathbb{C}^\alpha, h_2\in \mathbb{C}^\beta \}$ is $\alpha+\beta$, the above assumption on $\mathcal{K}$ is fulfilled whenever the codimension of $\mathcal{K}$ in $\mathbb{C}^{n(\alpha+\beta)+\zeta}$ is less than $\alpha+\beta$.
From this, it is easy to see that a CLP isometry is a generalized CLP isometry.

\begin{lemma}\label{MA}
A generalized CLP isometry is a unital positive linear isometry.
\end{lemma}
\begin{proof}
It is easy to see that a generalized CLP isometry is unital, positive, linear, and with norm at most $1$. 
Let $X$ be in $\mathscr{H}_n$, $\lambda\in \mathbb{R}$ an eigenvalue of $X$, and $h$ an eigenvector of $X$ for $\lambda$. 
Take vectors $h_1\in \mathbb{C}^\alpha, h_2\in \mathbb{C}^\beta$ such that $0\neq (h\otimes h_1) \oplus (\overline{h}\otimes h_2)\oplus 0_{\mathbb{C}^\zeta}\in \mathcal{K}$. Then we have 
\[
\begin{split}
&\quad \Phi(X)V^*((h\otimes h_1) \oplus (\overline{h}\otimes h_2)\oplus 0_{\mathbb{C}^\zeta}) \\
&= V^*((X\otimes I_\alpha)\oplus (X^t\otimes I_\beta)\oplus f(X))V V^*((h\otimes h_1) \oplus (\overline{h}\otimes h_2)\oplus 0_{\mathbb{C}^\zeta}) \\
&= V^*((X\otimes I_\alpha)\oplus (X^t\otimes I_\beta)\oplus f(X))((h\otimes h_1) \oplus (\overline{h}\otimes h_2)\oplus 0_{\mathbb{C}^\zeta}) \\
&=V^*((Xh\otimes h_1)\oplus (\overline{Xh}\otimes h_2) \oplus 0_{\mathbb{C}^\zeta})\\
&= \lambda V^*((h\otimes h_1) \oplus (\overline{h}\otimes h_2)\oplus 0_{\mathbb{C}^\zeta}). 
\end{split}
\]
Therefore, we see that $V^*((h\otimes h_1) \oplus (\overline{h}\otimes h_2)\oplus 0_{\mathbb{C}^\zeta})$ is an eigenvector of $\Phi (X)$ for the eigenvalue $\lambda$. 
Consequently, we have $\sigma(X)\subset \sigma(\Phi(X))$ for every $X\in \mathscr{H}_n$. 
Since $\lVert X\rVert =\max\{\lvert \lambda\rvert\mid \lambda\in \sigma(X)\}$ for every hermitian matrix $X$, we see that $\Phi $ is an isometry.
\end{proof}

We are going to exhibit examples of generalized CLP isometries, and then show that some of them are not CLP.
\begin{example}\label{nonclp}
Let $n$ be a positive integer. 
Let $\mathcal{K}:=\mathrm{span}\, \{h\otimes h\mid h\in \mathbb{C}^n\}$, which is a $k:=n(n+1)/2$-dimensional space. (One may identify $\mathcal{K}$ with the space of complex symmetric matrices by the correspondence  $h\otimes h\leftrightarrow hh^t$.)
Then, for each $h\in \mathbb{C}^n$ we have $h\otimes h\in \mathcal{K}$.
Thus, taking a linear isometry $V\colon \mathbb{C}^k\to \mathcal{K}\,(\subset \mathbb{C}^{n^2})$, we see that the mapping $\Theta_n\colon \mathscr{H}_n\ni X\mapsto V^*(X\otimes I_n)V\in \mathscr{H}_k$ is a generalized CLP isometry. 
Note that $\Theta_n$ is clearly a completely positive mapping.
\end{example}

\begin{lemma}
Fix $n$. 
There is no nontrivial subspace of $\mathbb{C}^k$ that is invariant simultaneously for $\Theta_n(X)$ for every $X\in \mathscr{H}_n$.
\end{lemma}
\begin{proof}
Let $e_1, e_2, \ldots, e_n$ denote the standard basis of $\mathbb{C}^n$. Set $f_{jk}= e_j\otimes e_k+e_k\otimes e_j$ and $g_{jk}=f_{jk}/\lVert f_{jk}\rVert$, $1\leq j\leq k\leq n$. It is a basic fact that $(g_{jk})_{1\leq j\leq k\leq n}$ forms an orthonormal basis of $\mathcal{K}=\mathrm{span}\, \{h\otimes h\mid h\in \mathbb{C}^n\}$. 
Let $P_l\in \mathscr{H}_n$ denote the projection onto $\mathbb{C}e_l$, $1\leq l\leq n$. 
Then, for $1\leq j\leq k\leq n$ and $1\leq j'\leq k'\leq n$, we have
\[
\begin{split}
\langle (P_l\otimes I_n)f_{jk}, f_{j'k'}\rangle &= \langle (P_l\otimes I_n)(e_{j}\otimes e_{k}+e_{k}\otimes e_{j}), e_{j'}\otimes e_{k'}+e_{k'}\otimes e_{j'}\rangle\\
&=\langle P_le_{j}\otimes e_{k}+P_le_{k}\otimes e_{j}, e_{j'}\otimes e_{k'}+e_{k'}\otimes e_{j'}\rangle.
\end{split}
\] 
\begin{itemize}
\item This is $0$ if $l\notin \{j,k\}$. 
\item If $l=j\neq k$, then this is equal to 
\[
\langle e_{j}\otimes e_{k}, e_{j'}\otimes e_{k'}+e_{k'}\otimes e_{j'}\rangle =\delta_{jj'}\delta_{kk'}+\delta_{jk'}\delta_{kj'}.
\]
If $\delta_{jk'}\delta_{kj'}\neq 0$, then $l=j=k'$ and $k=j'$ which together with $j\leq k$ and $j'\leq k'$ implies $j=k$, a contradiction. 
Thus we have 
\[
\delta_{jj'}\delta_{kk'}+\delta_{jk'}\delta_{kj'}=\delta_{jj'}\delta_{kk'}
\]
in this case.
\item Similarly, if $l=k\neq j$, then this is again equal to $\displaystyle \delta_{jj'}\delta_{kk'}$.
\item If $l=j=k$, then this is equal to 
\[
2\langle e_{j}\otimes e_{j}, e_{j'}\otimes e_{k'}+e_{k'}\otimes e_{j'}\rangle =4\delta_{jj'}\delta_{jk'}.
\]
\end{itemize}
Let us identify $\mathcal{K}$ with $\mathbb{C}^k$ using $V\colon \mathbb{C}^k\to \mathcal{K}$.
Then, with respect to the orthonormal basis $(g_{jk})_{1\leq j\leq k\leq n}$, $\Theta_n(P_l)$ may be considered as a diagonal matrix. Moreover, the diagonal entry corresponding to $1\leq j\leq k\leq n$ is equal to 
\begin{itemize}
\item $0$ if $l\notin \{j,k\}$, 
\item $1/2$ if $l=j\neq k$ or $l=k\neq j$, and
\item $1$ if $l=j=k$.
\end{itemize}
Using this, it is not hard to see that the algebra of all diagonal matrices is generated by $\Theta_n(P_l)$, $1\leq l\leq n$.
Therefore, a subspace which is invariant for $\Theta_n(X)$ for every $X\in \mathscr{H}_n$ needs to be a linear span of a subset of $\{f_{jk}\mid 1\leq j\leq k\leq n\}$.

Now, assume that $f_{jk}$ belongs to such a subspace but $f_{j'k}$ does not. Then we have $j\neq j'$.
Let $Q\in \mathscr{H}_n$ be the matrix whose $(j,j')$ and $(j',j)$-entries are $1$ and all other entries are $0$. 
Then $\langle(Q\otimes I_n)f_{jk}, f_{j'k}\rangle$ needs to be equal to $0$, but we have
\[
\begin{split}
\langle(Q\otimes I_n)f_{jk}, f_{j'k}\rangle &= \langle Qe_{j}\otimes e_{k}+Qe_{k}\otimes e_{j}, e_{j'}\otimes e_{k}+e_{k}\otimes e_{j'}\rangle\\
&= \langle e_{j'}\otimes e_{k}+Qe_{k}\otimes e_{j}, e_{j'}\otimes e_{k}+e_{k}\otimes e_{j'}\rangle\neq 0,
\end{split}
\]
a contradiction. Thus, if $f_{jk}$ belongs to such a subspace, then so does $f_{j'k}$ for every $j'$. Similarly, if $f_{jk}$ belongs to such a subspace, then so does $f_{jk'}$ for every $k'$, Thus such a subspace needs to be trivial. 
\end{proof}

\begin{proposition}
If $n\geq 3$ and $n$ is odd, then $\Theta_n$ is not CLP.
\end{proposition}
\begin{proof}
Assume that $\Theta_n$ is CLP, and that $\Theta_n$ is of the form \eqref{clp2}. 
If $k>n(\alpha+\beta)-\delta$, then it is easy to see that there is a nontrivial subspace of $\mathbb{C}^k$ that is invariant simultaneously for $\Theta_n (X)$ for every $X\in \mathscr{H}_n$. 
Therefore, the preceding lemma implies $n(n+1)/2=k=n(\alpha+\beta)-\delta$.
If $\delta=0$, then we have $\alpha+\beta\geq 2$, in which case we again see that there is a nontrivial invariant subspace, so $\delta>0$. 
Consequently, we get $k=n(\alpha+\beta)-\delta$ with $\alpha+\beta\geq 2$ and $\delta\in \{1,2,\ldots, \alpha+\beta-1\}$. As $k=n(n+1)/2$, we get $n(n+1)/2=n(\alpha+\beta)-\delta$ and hence $\delta =n\lambda\in \{1,2,\ldots, \lambda+\frac{n-1}{2}\}$, where $\lambda=\alpha+\beta -\frac{n+1}{2}$ is a positive integer. This is absurd.
\end{proof}

Thus we get infinitely many examples of non-CLP isometries.
Using the same idea, we will see that we may construct examples of non-CLP isometries on the full matrix algebras.

\begin{lemma}
Let $\Phi\colon \mathbb{M}_n\to \mathbb{M}_k$ be a CLP isometry. 
Assume that $\Phi(I_n)$ is a projection of rank $l$ whose range is $\mathcal{H}\subset \mathbb{C}^k$. 
Let $W\colon \mathbb{C}^l\to \mathcal{H}\subset \mathbb{C}^k$ be a linear isometry. 
Then the mapping $\Psi\colon \mathbb{M}_n\ni X\mapsto W^* \Phi(X) W\in \mathbb{M}_l$ restricts to a CLP isometry from $\mathscr{H}_n$ to $\mathscr{H}_l$.
\end{lemma}
\begin{proof}
The fact that the mapping $\Psi$ restricted to $\mathscr{H}_n$ is a unital positive linear isometry is a consequence of \cite[Theorem 2.1]{CLP} (or the discussion in the proof of Lemma \ref{kadison} below).
We may assume that $\Phi$ is the mapping as in \eqref{clp}.
Then $I_l=\Psi(I_n)=W^*V_1(T_1^*T_2\oplus f(I_n))V_2^*W$. 
It follows from the assumption that $T_1^*T_2$ and $f(I_n)$ are partial isometries, and so are $P=W^*V_1(T_1^*T_2\oplus 0_{k-m})V_2^*W$ and $Q=W^*V_1(0_m\oplus f(I_n))V_2^*W$. 
Since the initial space of $P$ is orthogonal to that of $Q$, the equation $P+Q=I_l$ implies that both $P$ and $Q$ are projections.
Therefore, the linear mappings $V_3:=(T_2\oplus 0_{k-m})V_2^*WP$ and $V_4:=(T_1\oplus 0_{k-m})V_1^*WP$ need to be partial isometries whose initial spaces equal the range of $P$. Moreover,  $V_4^*V_3=P$  implies that $V:=V_3=V_4$. We also see from $\mathrm{rank}\, (I_{n(\alpha+\beta)}-T_1T_1^*)+\mathrm{rank}\,(I_{n(\alpha+\beta)}-T_2T_2^*)<\alpha+\beta$ and $\mathrm{rank}\, {P} = \mathrm{rank}\, W^*V_1(T_1^*T_2\oplus 0_{k-m})V_2^*W = \mathrm{rank}\, T_1^*T_2$ that the rank $r$ of $P$ is greater than $(n-1)(\alpha+\beta)$.
Hence, by identifying the range of $P$ with $\mathbb{C}^r$, we see that the mapping $\mathbb{M}_n\ni X\mapsto V^*((X\otimes I_\alpha)\oplus (X^t\otimes I_\beta)\oplus 0_{k-m})V$ is of the form \eqref{phi0}.
It follows that the mapping $\Psi$ restricted to $\mathscr{H}_n$ is of the form \eqref{clp2}.
\end{proof}

\begin{example}
Let $n$, $\mathcal{K}$, $k$, $V$, $\Theta_n$ be the same as in Example \ref{nonclp}, and take the orthogonal projection $P\in \mathbb{M}_{n^2}$ onto $\mathcal{K}$. 
Define $\Xi_n\colon \mathbb{M}_n\to \mathbb{M}_{n^2}$ by $\Xi_n(X):= P(X\otimes I_n)\in \mathbb{M}_{n^2}$, $X\in \mathbb{M}_n$. 
Then, for each $X\in \mathbb{M}_n$, we have 
\[
\lVert \Xi_n(X)\rVert^2= \lVert \Xi_n(X)\Xi_n(X)^*\rVert = \lVert P(XX^*\otimes I_n)P\rVert,
\]
which is equal to
\[
\lVert V^*P(XX^*\otimes I_n)PV\rVert = \lVert \Theta_n(XX^*) \rVert =\lVert XX^*\rVert =\lVert X\rVert^2.
\]
Thus $\Xi_n$ is an isometry. 

Assume that $\Xi_n$ is CLP. 
Then, using the fact that $\Xi_n(I_n)=P$ and 
the preceding lemma, we see that the mapping $\mathscr{H}_n\ni X\mapsto V^*P\Xi_n(X)V\in \mathscr{H}_k$ needs to be a CLP isometry. However, this mapping is equal to $\Theta_n$. Therefore, if $\Theta_n$ is not CLP, then $\Xi_n$ is not CLP, either. 
\end{example}

Consequently, we have solved all problems (1)--(4) in \cite[p.\ 15]{CLP} negatively.
Indeed, $\Xi_n$ gives a counterexample to (1), and the mapping $\Theta_n$ (extended linearly to a mapping from $\mathbb{M}_n$ to $\mathbb{M}_k$) gives a counterexample to (2)--(4) (see also the proof of \cite[Proposition 4.3]{CLP}).

\section{Preliminaries}\label{preliminaries}
In what follows, we assume that the reader is acquainted with basic facts about C$^*$- or von Neumann algebras as in \cite{KR} or \cite{T1}.
\subsection{Extreme points of the unit ball}
For a normed space $\mathcal{X}$, we use the symbol $\mathcal{X}^{\leq 1}:=\{x\in\mathcal{X}\mid \lVert x\rVert\leq 1\}$ for the closed unit ball of $\mathcal{X}$.  For a convex subset $C\subset \mathcal{X}$, let $\mathrm{ext}\, C$ denote the set of extreme points of $C$.
Kadison \cite{K} gave the following characterization of extreme points of the unit ball. Recall that $\mathcal{A}$ is a unital C$^*$-algebra. 
\begin{theorem}[Kadison {\cite[Theorem 1]{K}}]\label{extreme}
An element $u\in \mathcal{A}$ lies in $\mathrm{ext}\, \mathcal{A}^{\leq 1}$ if and only if $u$ is a partial isometry satisfying $(1-uu^*)\mathcal{A}(1-u^*u)=\{0\}$. 
\end{theorem}
An \emph{atom} of a C$^*$-algebra is a projection $p\in \mathcal{P}(\mathcal{A})$ such that $\dim p\mathcal{A}p=1$. 
We will frequently use the following fact. Let $a\in \mathcal{A}_{sa}$.
If $f\colon \mathbb{R}\to \mathbb{R}$ restricts to a continuous function on $\sigma(a)$, by functional calculus we get an operator $f(a)$ which belongs to $\mathcal{A}_{sa}$.
For a subset $S\subset \mathbb{R}$, we use the symbol $\chi_{S}$ for the indicator function ($\chi_{S}(t)=1$ if $t\in S$ and $\chi_{S}(t)=0$ otherwise). Therefore, if $\#\sigma(a)<\infty$, then $\chi_S(a)$ is defined and it is a projection in $\mathcal{A}$.
We give a variant of Theorem \ref{extreme}.

\begin{proposition}\label{vvv}
Let $\mathcal{A}_1\subset \mathcal{A}$ be a linear subspace of finite codimension, and $v\in \mathrm{ext}\, \mathcal{A}_1^{\leq 1}$. Then there are partial isometries $v_0$, $v_1,v_2,\ldots, v_m$ in $\mathcal{A}$ having pairwise orthogonal initial spaces and pairwise orthogonal final spaces, and real numbers $1>c_1\geq c_2\geq\cdots\geq c_m>0$ such that  $v_1^*v_1$, $v_2^*v_2$, \ldots, $v_m^*v_m$ are atoms and $v=v_0+c_1v_1+c_2v_2+\cdots+c_mv_m$.
Moreover, the linear space $(1-v_0v_0^*)\mathcal{A}(1-v_0^*v_0)$ is finite-dimensional.
\end{proposition}
\begin{proof}
The proof here is inspired by \cite[Proof of Theorem 1]{K}.
We first prove that the spectrum $\sigma(v^*v)$ consists of finitely many points. 
Fix any continuous function $f\colon [0,1]\to [0,\infty)$ satisfying $0<f(t)\leq t^{-1/2}-1$ for every $t\in (0,1)$. 
If $\sigma(v^*v)$ is an infinite set, then the real linear space 
\[
\{vf(v^*v)g(v^*v)\mid g\colon [0,1]\to \mathbb{R}\text{ is a continuous function}\}\subset \mathcal{A}
\]
is infinite-dimensional.
Since $\mathcal{A}_1\subset \mathcal{A}$ is finite-codimensional, one may find a continuous function $g\colon [0,1]\to\mathbb{R}$ satisfying $0\neq vf(v^*v)g(v^*v)\in \mathcal{A}_1$.  By scaling, we may assume that $g([0,1])\subset [-1,1]$. 
Then we have 
\[
\begin{split}
\lVert v\pm vf(v^*v)g(v^*v)\rVert^2 &=\lVert (1\pm f(v^*v)g(v^*v))v^*v(1\pm f(v^*v)g(v^*v))\rVert\\
&= \lVert v^*v(1\pm f(v^*v)g(v^*v))^2\rVert\\
&\leq  \lVert v^*v(1+f(v^*v))^2\rVert\\
&\leq \sup_{t\in [0,1]} \lvert t(1+f(t))^2\rvert \leq 1.
\end{split}
\]
This contradicts the assumption that $v\in \mathrm{ext}\, \mathcal{A}_1^{\leq 1}$. 
Therefore, $\sigma(v^*v)$ is a finite set. 

If $e=\chi_{(0,1)}(\lvert v\rvert)$ satisfies $\dim e\mathcal{A}e=\infty$, taking an element $0\neq vx\in ve\mathcal{A}e\cap \mathcal{A}_1$ with sufficiently small norm, we get $\lVert v\pm vx\rVert = \lVert \lvert v\rvert (1\pm x)\rVert\leq 1$ and thus $v\pm vx\in \mathcal{A}_1^{\leq 1}$, a contradiction. Hence we obtain $\dim e\mathcal{A}e<\infty$, which means that $e$ is a finite sum of atoms of $\mathcal{A}$. 
Thus there are pairwise orthogonal projections $p_0$, $p_1, p_2, \ldots, p_m\in \mathcal{P}(\mathcal{A})$ and real numbers $1>c_1\geq c_2\geq\cdots\geq c_m>0$ such that $p_1, p_2, \ldots, p_m$ are atoms in $\mathcal{A}$ and $v^*v=p_0+c_1^2p_1+c_2^2p_2+\cdots+c_m^2p_m$. 

Consequently, we have $\lvert v\rvert=p_0+c_1p_1+c_2p_2+\cdots+c_mp_m$. 
Therefore, setting $w=v(p_0+c_1^{-1}p_1+c_2^{-1}p_2+\cdots+c_m^{-1}p_m)\in \mathcal{A}$, we see that $w$ is a partial isometry and $v=w\lvert v\rvert$ is the polar decomposition of $v$. Hence we get partial isometries $v_0=wp_0$, $v_1=wp_1$, $v_2=wp_2$, \ldots, $v_m=wp_m$ having pairwise orthogonal initial spaces and pairwise orthogonal final spaces with $v_0^*v_0=p_0$, $v_1^*v_1=p_1$, $v_2^*v_2=p_2$, \ldots, $v_m^*v_m=p_m$ and $v=v_0+c_1v_1+c_2v_2+\cdots+c_mv_m$. 

If the linear space $(1-v_0v_0^*)\mathcal{A}(1-v_0^*v_0)$ is infinite-dimensional, then one may find an element $a$ with $\lVert a\rVert = 1-c_1$ belonging to this set and also to $\mathcal{A}_1$. (In case $v=v_0$, set $c_1=0$.) 
Set $w=c_1v_1+c_2v_2+\cdots+c_mv_m$. Then the ranges of $v_0$ and $w\pm a$ (resp.\, $v_0^*$ and $(w\pm a)^*$) are orthogonal. 
Therefore, we have $\lVert v\pm a\rVert =\lVert v_0+w\pm a\rVert =\max\{\lVert v_0\rVert, \lVert w\pm a\rVert\}\leq 1$.
This once again contradicts the assumption that $v\in \mathrm{ext}\, \mathcal{A}_1^{\leq 1}$.
\end{proof}

Let us also give a self-adjoint variant. 

\begin{proposition}\label{ppp}
Let $\mathcal{A}_1\subset \mathcal{A}_{sa}$ be a linear subspace of finite codimension, and $v\in \mathrm{ext}\, \mathcal{A}_1^{\leq 1}$. 
Then $\# \sigma(v)<\infty$. Moreover, $\chi_{(-1,1)}(v)$ is the sum of finitely many atoms in $\mathcal{A}$.
\end{proposition}
Although the idea of the proof is almost the same as (or much easier than) Proposition \ref{vvv}, we give a full proof for completeness. 
\begin{proof}
We first prove that the spectrum $\sigma(v)$ consists of finitely many points. 
Define $f\colon [-1,1]\to [0,1]$ by $f(t)=1-\lvert t\rvert$ , $t\in [-1,1]$. 
If $\sigma(v)$ is an infinite set, then the real linear space 
\[
\{f(v)g(v)\mid g\colon [-1,1]\to \mathbb{R}\text{ is a continuous function}\}\subset \mathcal{A}_{sa}
\]
is infinite-dimensional.
Since $\mathcal{A}_1\subset \mathcal{A}_{sa}$ is finite-codimensional, one may find a continuous function $g\colon [-1,1]\to\mathbb{R}$ satisfying $0\neq f(v)g(v)\in \mathcal{A}_1$.  By scaling, we may assume that $g([-1,1])\subset [-1,1]$. 
Then we have 
\[
\lVert v\pm f(v)g(v)\rVert \leq \sup_{t\in [-1,1]} \lvert t\pm f(t)g(t)\rvert \leq 1.
\]
This contradicts the assumption that $v\in \mathrm{ext}\, \mathcal{A}_1^{\leq 1}$. 
Therefore, $\sigma(v)$ is a finite set. 
Set $e=\chi_{(-1,1)}(v)$. If $\dim e\mathcal{A}_{sa}e=\infty$, taking an element $0\neq x\in e\mathcal{A}_{sa}e\cap \mathcal{A}_1$ with sufficiently small norm, we get $v\pm x\in \mathcal{A}_1^{\leq 1}$, a contradiction. Hence we obtain $\dim e\mathcal{A}_{sa}e<\infty$, which means that $e$ is a finite sum of atoms of $\mathcal{A}$. 
\end{proof}

We remark that the result \cite[Theorem 3.1]{AW} by Akemann and Weaver is a variant of this proposition.

\subsection{Properties of positive mappings, isometries, and Jordan $^*$-homomorphisms}
We begin with a well-known fact. For readability, we give its proof.
\begin{lemma}\label{restrict}
Let $\Phi\colon \mathcal{A}\to \mathcal{B}$ be a unital complex linear mapping with $\lVert \Phi\rVert\leq 1$. Then $\Phi$ is a positive linear mapping (that is, if $a\in \mathcal{A}$ is positive, then $\Phi(a)\in \mathcal{B}$ is positive), and in particular, we have $\Phi(\mathcal{A}_{sa})\subset \mathcal{B}_{sa}$.
\end{lemma}
\begin{proof}
Let $a\in \mathcal{A}$ be a positive operator with $\min\sigma(a)=0$ and $\max\sigma(a)=1$. 
Decompose $\Phi(a)$ as $b_1+ib_2$ with $b_1, b_2\in \mathcal{B}_{sa}$.  
For every real number $t\in \mathbb{R}$, we have  
\[
\sqrt{1+t^2}=\lVert a+it\rVert\geq \lVert \Phi(a)+it\rVert=\lVert b_1+i(b_2+t)\rVert\geq \lVert b_2+t\rVert.
\]
If $0<\lambda\in \sigma(b_2)$, then we have $\lVert b_2+t\rVert\geq \lambda+t$ for every $t>0$, so $\sqrt{1+t^2}\geq \lambda+t$. Taking the square, we get $1\geq \lambda^2+2\lambda t$ for every $t>0$, which is absurd. 
If $0>\lambda\in \sigma(b_2)$, then $\lVert b_2+t\rVert\geq -\lambda-t$ for every $t<0$, and we again get to a contradiction. 
Therefore, we get $b_2=0$.
Moreover, the inequality \[
1=\lVert 2a-1\rVert \geq \lVert \Phi(2a-1)\rVert=\lVert 2b_1-1\rVert
\]
shows that $b_1$ is a positive operator. Thus $\Phi(a)\in \mathcal{B}$ is positive.

Since every positive operator in $\mathcal{A}$ can be written as the sum of some nonnegative scalar multiple of $a$ as above and some nonnegative scalar multiple of $1$, we get the desired conclusion.
\end{proof}

On the other hand, we should beware of the fact that the image of a unital real linear mapping from $\mathcal{A}_{sa}$ into $\mathcal{B}$ with norm at most $1$ is not necessarily contained in $\mathcal{B}_{sa}$, see \cite[Remark 2.2]{CLP}.

The following is a theorem by Kadison. (In the statement of \cite[Theorem 1]{K2}, a mapping from $\mathcal{A}$ to $\mathcal{B}$ is considered, but actually the following statement is verified in the proof.)

\begin{theorem}[Kadison's generalized Schwarz inequality {\cite[Theorem 1]{K2}}, see also {\cite[Theorem 1.3.1]{S2}}]\label{ks}
Let $\Phi\colon \mathcal{A}_{sa}\to \mathcal{B}_{sa}$ be a positive real linear mapping with $\lVert \Phi\rVert\leq 1$. 
If $a\in \mathcal{A}_{sa}$, then $\Phi(a^2)\geq \Phi(a)^2$. 
\end{theorem}

For a Hilbert space $\mathcal{H}$, let $\mathbb{B}(\mathcal{H})$ denote the von Neumann algebra of all bounded linear operators on $\mathcal{H}$. The next fact is given by Kadison and St\o rmer.

\begin{theorem}[{\cite[Theorem 3.3]{S}}, see also {\cite[Theorem 10]{K}}]\label{jhomo}
Let $J\colon \mathcal{A}\to \mathbb{B}(\mathcal{H})$ be a Jordan $^*$-homomorphism. Then there is a central projection $p\in J(\mathcal{A})''$ such that the mapping $\mathcal{A}\ni a\mapsto pJ(a)\in \mathbb{B}(\mathcal{H})$ is a $^*$-homomorphism and the mapping $\mathcal{A}\ni a\mapsto (1-p)J(a)\in \mathbb{B}(\mathcal{H})$ is a $^*$-antihomomorphism.
\end{theorem}

\begin{corollary}\label{kernel}
The kernel of a Jordan $^*$-homomorphism $J\colon \mathcal{A}\to \mathcal{B}$ is a closed ideal in $\mathcal{A}$.
\end{corollary}
\begin{proof}
One may consider $\mathcal{B}$ as a C$^*$-subalgebra of $\mathbb{B}(\mathcal{H})$.
Thus the kernel of $J$ is the intersection of the kernel of some $^*$-homomorphism and that of some $*$-antihomomorphism, so it is a closed ideal. 
\end{proof}

We call a linear subspace $\mathcal{A}_1\subset \mathcal{A}_{sa}$ satisfying $a^2\in \mathcal{A}_1$ for any $a\in \mathcal{A}_1$ (or equivalently, $a_1a_2+a_2a_1\in \mathcal{A}_1$ for any pair $a_1,a_2\in \mathcal{A}_1$) a \emph{Jordan subalgebra} of $\mathcal{A}_{sa}$. 
We will also need the following fact, which is essentially due to Broise (see \cite[Proposition 2.1.7]{S2}). We give a proof for completeness.
\begin{lemma}\label{broise}
Let $\Phi\colon \mathcal{A}_{sa}\to \mathcal{B}_{sa}$ be a positive real linear mapping with $\lVert \Phi\rVert\leq 1$. 
Then the set $\{a\in \mathcal{A}_{sa}\mid \Phi(a)^2=\Phi(a^2)\}$ is a Jordan subalgebra of $\mathcal{A}_{sa}$. 
\end{lemma}
\begin{proof}
Let $a\in \mathcal{A}_{sa}$ be an element satisfying $\Phi(a)^2=\Phi(a^2)$, and let $x\in \mathcal{A}_{sa}$. 
By Kadison's generalized Schwarz inequality (Theorem \ref{ks}), for every $t\in \mathbb{R}$, we have
\[
\begin{split}
t(\Phi(a)\Phi(x)+\Phi(x)\Phi(a))&= \Phi(ta+x)^2 -t^2\Phi(a)^2-\Phi(x)^2\\
&\leq \Phi((ta+ x)^2) -t^2\Phi(a)^2-\Phi(x)^2\\
&=(t^2\Phi(a^2) + t\Phi(ax+xa)+\Phi(x^2))-t^2\Phi(a)^2-\Phi(x)^2\\
&=t\Phi(ax+xa) +\Phi(x^2)-\Phi(x)^2.
\end{split}
\]
Thus we get $\Phi(ax+xa)=\Phi(a)\Phi(x)+\Phi(x)\Phi(a)$. 
In particular, we have 
\[
\Phi(a^3)=\frac{1}{2}(\Phi(a)\Phi(a^2)+\Phi(a^2)\Phi(a))=\Phi(a)^3
\]
and 
\[
\Phi((a^2)^2)=\Phi(a^4)=\frac{1}{2}(\Phi(a)\Phi(a^3)+\Phi(a^3)\Phi(a))=\Phi(a)^4=\Phi(a^2)^2.
\]
On the other hand, if we additionally assume that $\Phi(x)^2=\Phi(x^2)$, then
\[
\Phi(a+x)^2=\Phi(a)^2+\Phi(a)\Phi(x)+\Phi(x)\Phi(a)+\Phi(x)^2= \Phi(a^2)+\Phi(ax+xa)+\Phi(x^2)= \Phi((a+x)^2).
\]
By what we have shown, we see that the set $\{a\in \mathcal{A}_{sa}\mid \Phi(a)^2=\Phi(a^2)\}$ is a Jordan subalgebra of $\mathcal{A}_{sa}$. 
%Moreover, we obtain
%\[
%\begin{split}
%\Phi(2axa)&= \Phi((a(ax+xa) + (ax+xa)a)-(a^2x+xa^2))\\
%&= \Phi(a)\Phi(ax+xa)+ \Phi(ax+xa)\Phi(a)-(\Phi(a^2)\Phi(x)+\Phi(x)\Phi(a^2))\\ 
%&= \Phi(a)(\Phi(a)\Phi(x)+\Phi(x)\Phi(a))+ (\Phi(a)\Phi(x)+\Phi(x)\Phi(a))\Phi(a)-(\Phi(a)^2\Phi(x)+\Phi(x)\Phi(a)^2)\\ 
%&= 2\Phi(a)\Phi(x)\Phi(a)
%\end{split}
%\]
%and thus $\Phi(axa)=\Phi(a)\Phi(x)\Phi(a)$. Therefore, if we additionally assume that $\Phi(x)^2=\Phi(x^2)$, then 
%\[
%\begin{split}
%\Phi((ax+xa)^2)&= \Phi(a(xax)+(xax)a+ax^2a+xa^2x)\\
%&= \Phi(a)\Phi(xax)+\Phi(xax)\Phi(a) +\Phi(a)\Phi(x^2)\Phi(a) + \Phi(x)\Phi(a^2)\Phi(x)\\
%&= \Phi(a)\Phi(x)\Phi(a)\Phi(x)+\Phi(x)\Phi(a)\Phi(x)\Phi(a) +\Phi(a)\Phi(x)^2\Phi(a) + \Phi(x)\Phi(a)^2\Phi(x)\\
%&= (\Phi(a)\Phi(x)+\Phi(x)\Phi(a))^2\\
%&= \Phi(ax+xa)^2
%\end{split}
%\]
%By what we have shown, we see that the set $\{a\in \mathcal{A}_{sa}\mid \Phi(a)^2=\Phi(a^2)\}$ is a Jordan subalgebra of $\mathcal{A}_{sa}$. 
\end{proof}

The next lemma is a variant of  \cite[Theorem 2.1]{CLP}. 

\begin{lemma}\label{psi}
Let $\Psi\colon \mathcal{A}_{sa}\to \mathcal{B}_{sa}$ be a unital real linear mapping. Then the following are equivalent.  
\begin{itemize} 
\item $\Psi$ is an isometry.
\item $\Psi$ is an order embedding. 
\item $\Psi$ satisfies $\max \sigma(\Psi(a))=\max\sigma(a)$ and $\min \sigma(\Psi(a))=\min\sigma(a)$ for every $a\in \mathcal{A}_{sa}$.
\end{itemize}
\end{lemma}
\begin{proof}
For an element $a\in \mathcal{A}_{sa}$ with $\min\sigma(a) = m$ and $\max\sigma(a)=M$, we have $\lVert a\rVert =\max\{\lvert m\rvert, \lvert M\rvert\}$. Using this fact, for $t\in \mathbb{R}$, we obtain 
\[
\lVert a-t\rVert =\left\lvert t-\frac{m+M}{2}\right\rvert + \frac{M-m}{2}.
\]

Assume that $\Psi$ is an isometry.
Since $\Psi$ is unital, we have 
\[
\lVert \Psi(a)-t\rVert =\lVert a-t\rVert =\left\lvert t-\frac{m+M}{2}\right\rvert + \frac{M-m}{2}
\]
for every $t\in \mathbb{R}$.
This in turn implies $\min\sigma(\Psi(a)) = m$ and $\max\sigma(\Psi(a))=M$. Thus the third item holds. 

Assume that $\Psi$ is an order embedding. 
Since $\Psi$ is unital, we have 
\[
\min\sigma(a)=\max\{t\in \mathbb{R}\mid t\leq a\} =\max\{t\in \mathbb{R}\mid t\leq \Psi(a)\}=\min\sigma(\Psi(a))
\]
and
\[
\max\sigma(a)=\min\{t\in \mathbb{R}\mid a\leq t\} =\min\{t\in \mathbb{R}\mid \Psi(a)\leq t\}=\max\sigma(\Psi(a)).
\]
Thus the third item holds.

Conversely, the third item clearly implies $\lVert \Psi(a)\rVert =\lVert a\rVert$, and $a\geq 0\iff \Psi(a)\geq 0$ for every $a\in \mathcal{A}_{sa}$, which together with the linearity implies that $\Psi$ is an isometric order embedding.
\end{proof}

The next lemma is probably well-known among experts. 
\begin{lemma}\label{isometry}
If $J\colon \mathcal{A}\to\mathcal{B}$ is an injective Jordan $^*$-homomorphism, then $J$ is an isometry.
\end{lemma}
\begin{proof}
One may consider $J$ as an injective unital Jordan $^*$-homomorphism from $\mathcal{A}$ to $J(1)\mathcal{B} J(1)$. Thus we may assume that $J$ is unital without loss of generality. We may also assume that $\mathcal{B}=\mathbb{B}(\mathcal{H})$.
Since $J(a^2)=J(a)^2$ for every $a\in \mathcal{A}_{sa}$, we see that $J$ restricts to unital linear order embedding from $\mathcal{A}_{sa}$ to $\mathbb{B}(\mathcal{H})_{sa}$. By the preceding lemma, we get $\lVert J(a)\rVert =\lVert a\rVert$ for every $a\in \mathcal{A}_{sa}$. 
Take the projection $p\in \mathbb{B}(\mathcal{H})$ as in Theorem \ref{jhomo}. 
For $a\in \mathcal{A}$, we have 
\[
\begin{split}
\lVert J(a)\rVert^2 &= \max\{\lVert pJ(a)\rVert^2, \lVert(1-p)J(a) \rVert^2\} \\
&= \max\{\lVert (pJ(a))^*(pJ(a))\rVert, \lVert((1-p)J(a))((1-p)J(a))^* \rVert\} \\
&= \lVert pJ(a)^*J(a)+(1-p)J(a)J(a)^*\rVert\\
&= \lVert J(a^*a)\rVert.
\end{split}
\] 
Since $a^*a$ is self-adjoint, we obtain $\lVert J(a^*a)\rVert=\lVert a^*a\rVert=\lVert a\rVert^2$. Thus $\lVert J(a)\rVert=\lVert a\rVert$.
\end{proof}

Lastly, we give a variant of \cite[Corollary 2]{RD} by Russo and Dye who considered the case where $G=\mathcal{U}(\mathcal{A})$.
\begin{theorem}\label{russodye}
Let $G\subset \mathcal{U}(\mathcal{A})$ be a nonempty open subset. Let $\Phi \colon \mathcal{A}\to \mathcal{B}$ be a linear mapping satisfying $\lVert \Phi\rVert\leq 1$ and $\Phi (u)\in \mathcal{U}(\mathcal{B})$ for every $u\in G$. 
Fix an element $u_0\in G$ and set $\Phi _0(a) := \Phi(u_0)^*\Phi(u_0a)$ for $a\in \mathcal{A}$. 
Then $\Phi_0\colon \mathcal{A}\to \mathcal{B}$ is a unital Jordan $^*$-homomorphism, and $\Phi(\mathcal{U}(\mathcal{A})) \subset \mathcal{U}(\mathcal{B})$.
\end{theorem}
\begin{proof}
The method of the proof here is almost the same as Russo and Dye's proof of \cite[Corollary 2]{RD}.
Observe that $\Phi _0\colon \mathcal{A}\to \mathcal{B}$ is a unital linear mapping with $\lVert \Phi_0\rVert\leq 1$ (which is automatically positive by Lemma \ref{restrict}) satisfying $\Phi _0(u)\in \mathcal{U}(\mathcal{B})$ for every unitary $u$ in $\mathcal{A}$ that is close enough to $1\in \mathcal{A}$. 
Let $u\in \mathcal{U}(\mathcal{A})$ be close enough to $1$. 
Then Theorem \ref{ks} implies that $\Phi _0((u+u^*)^2)\geq \Phi_0(u+u^*)^2$, hence $\Phi _0(u^2)+\Phi_0(u^2)^*\geq \Phi_0(u)^2+\Phi_0(u^*)^2$. Since $\Phi_0(iu)=i\Phi_0(u)$ is also a unitary, the same inequality holds when we replace $u$ by $iu$, so we have $-\Phi_0(u^2)-\Phi_0(u^2)^*\geq -\Phi_0(u)^2-\Phi_0(u^*)^2$. These two inequalities show $\Phi_0((u+u^*)^2)= \Phi_0(u+u^*)^2$. 
If a positive operator $a\in \mathcal{A}$ satisfies $a\leq 2$ and $a$ is close enough to $2$, then $a=u+u^*$ for some $u\in \mathcal{U}(\mathcal{A})$ that is close to $1$, so we have $\Phi _0(a^2)=\Phi_0(a)^2$. 
By Lemma \ref{broise}, the same equality holds for every $a\in \mathcal{A}_{sa}$, and moreover, for every $a\in \mathcal{A}$ by routine calculations.
Consequently, $\Phi _0$ is a unital Jordan $^*$-homomorphism. 

It follows that $\Phi _0$ sends $\mathcal{U}(\mathcal{A})$ into $\mathcal{U}(\mathcal{B})$. 
Indeed, for every $u\in \mathcal{U}(\mathcal{A})$, we have 
\[
2=\Phi_0(2)=\Phi_0(uu^*+u^*u)=\Phi_0(u)\Phi_0(u)^*+\Phi_0(u)^*\Phi_0(u),
\]
which together with $\lVert \Phi_0(u)\rVert\leq 1$ implies $\Phi _0(u)\Phi_0(u)^*=1=\Phi_0(u)^*\Phi_0(u)$.
Since $\Phi(u_0)\in \mathcal{B}$ is unitary, we get $\Phi(\mathcal{U}(\mathcal{A})) \subset \mathcal{U}(\mathcal{B})$.
\end{proof}

\subsection{Petty subsets}
In the subsequent sections, we frequently encounter the situation where we need to think of finite-dimensional ideals of C$^*$-algebras. 
For such cases, we introduce the following concept. 
We say a subset $S\subset \mathcal{A}$ is \emph{petty} if the ideal of $\mathcal{A}$ generated by $S$ is finite-dimensional. 
An element $a\in \mathcal{A}$ is said to be petty if $\{a\}\subset \mathcal{A}$ is petty.
Let us give some basic facts about finite-dimensional ideals and petty subsets.

\begin{lemma}\label{fdideal}
Let $\mathcal{I}\subset \mathcal{A}$ be a finite-dimensional ideal of $\mathcal{A}$. 
Then $\mathcal{I}$ is a closed $^*$-subalgebra of $\mathcal{A}$, and there is a central projection $e\in \mathcal{A}$ with $\mathcal{I}=e\mathcal{A}$.
\end{lemma}
\begin{proof}
The finite-dimensionality of $\mathcal{I}$ implies that $\mathcal{I}$ is closed, and it is a basic fact that a closed ideal of a C$^*$-algebra is a $^*$-subalgebra.
Let $e$ be the unit of the finite-dimensional C$^*$-algebra $\mathcal{I}$. 
Then it is readily seen that $e$ considered as an element of $\mathcal{A}$ is a projection. 
Moreover, for $a\in \mathcal{A}$, we have $ae, ea\in \mathcal{I}$ and hence $ae=eae=ea$, so $e$ is a central projection.
\end{proof}

\begin{lemma}\label{spetty}
If $S_1, S_2, \ldots, S_m$ are petty subsets of $\mathcal{A}$, then $\mathrm{span}\, (\bigcup_{j=1}^mS_j)\subset \mathcal{A}$ is also petty.  
\end{lemma}
\begin{proof}
This is a consequence of the facts that an ideal is a linear space and that the sum of finite-dimensional ideals is again a finite-dimensional ideal.
\end{proof}

\begin{lemma}\label{epetty}
Let $a\in \mathcal{A}$ be an element. Then $a$ is petty if and only if $\dim a\mathcal{A}<\infty$.
\end{lemma}
\begin{proof}
If $a$ is petty, then $a\mathcal{A}$ is contained in a finite-dimensional ideal of $\mathcal{A}$, so $\dim a\mathcal{A}<\infty$. 
Conversely, assume that $\dim a\mathcal{A}<\infty$ holds. If $\sigma(\lvert a\rvert)$ is an infinite set, then $\{af(\lvert a\rvert)\mid f\colon \sigma(\lvert a\rvert)\to \mathbb{R}\text{ is a continuous function}\}$ is infinite-dimensional, so we get a contradiction. Thus $\sigma(\lvert a\rvert)$ is a finite set. 
This implies that we may get the polar decomposition $a=v\lvert a\rvert$ inside $\mathcal{A}$.  
Set $f(0)=0$ and $f(t)=t^{-1}$, $t>0$. Then $f$ can be considered as a continuous function on $\sigma(\lvert a\rvert)$, and the element $f(\lvert a\rvert)\in \mathcal{A}$ satisfies $\lvert a\rvert f(\lvert a\rvert) =v^*v$. Thus we get
\[
vv^*\mathcal{A}\subset v\mathcal{A} =v\lvert a\rvert f(\lvert a\rvert)\mathcal{A}\subset v\lvert a\rvert\mathcal{A}=a\mathcal{A}. 
\]
Now, take a basis $a_1,a_2, \ldots, a_m$ of the finite-dimensional linear space $a\mathcal{A}$. 
Then every element in $\mathcal{A}a\mathcal{A} = \mathcal{A}vv^*a\mathcal{A} = (vv^*\mathcal{A})^*(a\mathcal{A})$ can be written as a linear combination of elements of the form $a_i^*a_j$, $1\leq i,j\leq m$. Thus $a$ is petty.
\end{proof}

\begin{lemma}\label{eepetty}
Let $e_1, e_2, \ldots, e_m\in \mathcal{A}$ be mutually orthogonal projections. 
Set $\mathcal{V}=\{a\in \mathcal{A}\mid e_1ae_1=e_2ae_2=\cdots=e_mae_m=0\}$, which is a linear subspace of $\mathcal{A}$. 
Assume that $\dim \mathcal{V}<\infty$ or $\dim \mathcal{V}\cap \mathcal{A}_{sa}<\infty$ holds. Then $\mathcal{V}\subset \mathcal{A}$ is petty.
Take the central projection $e$ such that the ideal generated by $\mathcal{V}$ equals $e\mathcal{A}$. Then $(1-e)e_1$, $(1-e)e_2$, \ldots, $(1-e)e_m$ are central projections in $\mathcal{A}$.
\end{lemma}
\begin{proof}
It is clear that $\dim \mathcal{V}<\infty$ if and only if $\dim \mathcal{V}\cap \mathcal{A}_{sa}<\infty$.
Observe that $\mathcal{V}$ is spanned by elements of $e_j\mathcal{A}(1-e_j)$, $(1-e_j)\mathcal{A}e_j$ ($1\leq j\leq m$), and $(1-e_1-e_2-\cdots-e_m)\mathcal{A}(1-e_1-e_2-\cdots-e_m)$. By Lemma \ref{spetty}, it suffices to show that $a\in \mathcal{A}$ is petty for elements belonging to each of the above sets. If $a\in e_j\mathcal{A}(1-e_j)$, then 
\[
a\mathcal{A}=e_ja(1-e_j)\mathcal{A}=e_ja(1-e_j)\mathcal{A}e_j+e_ja(1-e_j)\mathcal{A}(1-e_j)\subset e_ja\mathcal{V}+ e_j\mathcal{A}(1-e_j)=e_ja\mathcal{V}+\mathcal{V}, 
\]
so Lemma \ref{epetty} implies that $a$ is petty. Since $((1-e_j)\mathcal{A}e_j)^*=e_j\mathcal{A}(1-e_j)$, we see that every element of $(1-e_j)\mathcal{A}e_j$ is also petty. If $a\in (1-e_1-e_2-\cdots-e_m)\mathcal{A}(1-e_1-e_2-\cdots-e_m)$, then $a\mathcal{A}\subset \mathcal{V}$ and thus $a$ is petty. Thus we have shown that $\mathcal{V}\subset \mathcal{A}$ is petty.

It then follows from the definition of $e$ that $j\neq j'$ implies $r_j\mathcal{A}r_{j'}=0$, where $r_j=(1-e)e_j$. Moreover, $e=1-(r_1+r_2+\cdots+r_m)$ is a central projection. Thus, we obtain $r_j \mathcal{A}(1-r_j)=0$, so $(1-e)e_j=r_j$ is a central projections in $\mathcal{A}$.
\end{proof}

\section{Proof of the main theorem}\label{secmain}
In this section, we assume that $\Phi\colon \mathcal{A}\to \mathcal{B}$ is a linear isometry with finite corank.

\subsection{Isometries with finite corank}
\begin{lemma}\label{kadison}
Let $u\in \mathcal{A}$ be a unitary. Set $v=\Phi(u)$. Observe that $u\in \mathrm{ext}\, \mathcal{A}^{\leq 1}$ by Kadison's theorem (Theorem \ref{extreme}), which implies $v\in \mathrm{ext}\, \Phi(\mathcal{A})^{\leq 1}$, so we may apply Proposition \ref{vvv} and 
write $v=v_0+c_1v_1+c_2v_2+\cdots+c_mv_m$. Then $\{1-v_0^*v_0, 1-v_0v_0^*\}\subset \mathcal{B}$ is petty.
\end{lemma} 
\begin{proof}
The discussion below is a modification of that in the proofs of \cite[Theorem 7]{K} by Kadison and \cite[Theorem 2]{CLP} by Cheung, Li, and Poon.
By considering the mapping $\mathcal{A}\ni a\mapsto \Phi(ua)\in \mathcal{B}$ instead of $\Phi$, we may and will assume that $u=1$.
Let $a\in \mathcal{A}_{sa}$ be an element satisfying $\lVert a\rVert=1=\max\sigma(a)$. 
Then we have 
\[
2=\lVert 1+a\rVert =\lVert \Phi(1+a)\rVert =\lVert v+\Phi(a)\rVert.
\] 
We may assume that $\mathcal{B}$ acts on a Hilbert space $\mathcal{H}$. Then we may take a sequence of unit vectors $g_n, h_n\in \mathcal{H}$ satisfying 
$\langle (v+\Phi(a))g_n, h_n\rangle\to \lVert v+\Phi(a)\rVert =2$ as $n\to\infty$. 
This together with $\lVert v\rVert =\lVert \Phi(a)\rVert =1$ implies $\langle vg_n, h_n\rangle \to 1$ and  $\langle \Phi(a)g_n, h_n\rangle \to 1$. 
Setting $v_j^*v_j g_n= g_{n, j}$ and $v_jv_j^* h_n= h_{n, j}$, $j=0,1,\ldots, m$, we have
\[
\begin{split}
\langle vg_n, h_n\rangle &= \langle v_0g_n, h_n\rangle + \langle c_1v_1g_n, h_n\rangle + \cdots +\langle c_mv_mg_n, h_n\rangle\\
&= \langle v_0g_{n,0}, h_{n,0}\rangle + c_1\langle v_1g_{n,1}, h_{n,1}\rangle + \cdots +c_m\langle v_mg_{n,m}, h_{n,m}\rangle.
\end{split}
\]
Thus $\lvert \langle vg_n, h_n\rangle\rvert$ is at most 
\[
\sqrt{\lVert g_{n,0}\rVert^2 + c_1\lVert g_{n,1}\rVert^2 +\cdots+c_m\lVert g_{n,m}\rVert^2} \sqrt{\lVert h_{n,0}\rVert^2 + c_1\lVert h_{n,1}\rVert^2 +\cdots+c_m\lVert h_{n,m}\rVert^2}
\]
by the Cauchy--Schwarz inequality.
However, since the initial (resp. final) spaces of $v_0, v_1,\ldots, v_m$ are orthogonal, we see that $\lVert g_{n,0}\rVert^2 + \lVert g_{n,1}\rVert^2 +\cdots+\lVert g_{n,m}\rVert^2\leq \lVert g_n\rVert^2=1$ (resp. $\lVert h_{n,0}\rVert^2 + \lVert h_{n,1}\rVert^2 +\cdots+\lVert h_{n,m}\rVert^2\leq \lVert h_n\rVert^2=1$). 
Combining these facts together with $1>c_1\geq c_2\geq \cdots\geq c_m>0$, we see that $g_{n,1}, g_{n,2},\ldots, g_{n,m}$ and $h_{n,1}, h_{n,2},\ldots, h_{n,m}$ all converge to $0$, and $\lVert g_n-g_{n,0}\rVert\to 0$, $\lVert h_n-h_{n,0}\rVert\to 0$,  as $n\to \infty$. 
Therefore, we get $\langle vg_{n,0}, h_{n,0}\rangle=\langle v_0g_{n,0}, h_{n,0}\rangle \to 1$ and  $\langle \Phi(a)g_{n,0}, h_{n,0}\rangle \to 1$.
By
\[
0\leq \lVert v_0g_{n,0}-h_{n,0}\rVert^2= \lVert v_0g_{n,0}\rVert^2-2\mathrm{Re}\, \langle v_0g_{n,0}, h_{n,0}\rangle + \lVert h_{n,0}\rVert^2  
\]
and $\lVert v_0g_{n,0}\rVert, \lVert h_{n,0}\rVert\leq 1$, we get $\lVert v_0g_{n,0}-h_{n,0}\rVert \to 0$. 
Hence $\langle \Phi(a)g_{n,0}, v_0g_{n,0}\rangle \to 1$. 
Set $p_0=v_0^*v_0$.
Since $\lVert g_{n,0}\rVert\leq 1$ and $\langle \Phi(a)g_{n,0}, v_0g_{n,0}\rangle = \langle v_0^*\Phi(a)p_0g_{n,0}, g_{n,0}\rangle$, we get $\lVert v_0^*\Phi(a)p_0\rVert =1$. 
It follows that $\lVert v_0^*\Phi(\lambda a)p_0\rVert =\lvert \lambda\rvert$ for every $\lambda\in \mathbb{R}$. 
Therefore, we have shown in particular that $\lVert v_0^*\Phi(a)p_0\rVert =\lVert a\rVert$ for every $a\in \mathcal{A}_{sa}$.

Now, we consider the mapping $\Phi _0\colon \mathcal{A}\ni a\mapsto v_0^*\Phi(a)p_0\in \mathcal{B}$. 
We may and will think of $\Phi _0$ as a unital linear mapping from $\mathcal{A}$ to the C$^*$-algebra $p_0\mathcal{B}p_0$ with $\lVert \Phi_0\rVert\leq 1$. 
By Lemma \ref{restrict}, $\Phi _0$ is a positive mapping and in particular, $\Phi _0(\mathcal{A}_{sa})\subset (p_0\mathcal{B}p_0)_{sa}$. 
Moreover, the preceding paragraph implies that $\Phi _0$ restricted to $\mathcal{A}_{sa}$ is an isometry.
Therefore, the injectivity of $\Phi _0$ on $\mathcal{A}_{sa}$ shows that $\Phi _0$ is injective on $\mathcal{A}$ as well.
(However, beware of the fact that $\Phi _0$ is not necessarily isometric on $\mathcal{A}$.)
On the other hand, if $\mathcal{B}(1-p_0)$ is not finite-dimensional, then the mapping $\mathcal{B}\ni b\mapsto bp_0$ has infinite-dimensional kernel, which together with the finite-codimensionality of $\Phi (\mathcal{A})\subset \mathcal{B}$ implies that $\Phi _0$ is not injective, a contradiction. 
Thus $\mathcal{B}(1-p_0)$ is finite-dimensional, and so is $(1-p_0)\mathcal{B}=(\mathcal{B}(1-p_0))^*$. By Lemma \ref{epetty}, $1-p_0=1-v_0^*v_0\in \mathcal{B}$ is petty. Similarly, $1-v_0v_0^*\in \mathcal{B}$ is petty. It follows from Lemma \ref{spetty} that $\{1-v_0^*v_0, 1-v_0v_0^*\}\subset \mathcal{B}$ is petty.
\end{proof}

\begin{lemma}\label{qb}
There is a petty central projection $q\in \mathcal{B}$ satisfying $(1-q)\Phi(u)\in \mathcal{U}((1-q)\mathcal{B})$ for every $u\in \mathcal{U}(\mathcal{A})$.
\end{lemma}
\begin{proof}
Let $f$ be a central projection in $\mathcal{B}$ such that $f\mathcal{B}$ is $^*$-isomorphic to the full matrix algebra $\mathbb{M}_m$ for some $m$.
If $\{u\in \mathcal{U}(\mathcal{A})\mid f\Phi(u)\in \mathcal{U}(f\mathcal{B})\}\subset \mathcal{U}(\mathcal{A})$ has nonempty interior, then Theorem \ref{russodye} implies that $f\Phi(u)\in \mathcal{U}(f\mathcal{B})$ for every $u\in \mathcal{U}(\mathcal{A})$. Therefore, the closed subset $\{u\in \mathcal{U}(\mathcal{A})\mid f\Phi(u)\in \mathcal{U}(f\mathcal{B})\}\subset \mathcal{U}(\mathcal{A})$ is either nowhere dense or equal to $\mathcal{U}(\mathcal{A})$.  

Assume that there is a sequence of distinct central projections $f_j$, $j\geq 1$, in $\mathcal{B}$ such that each $f_j\mathcal{B}$ is $^*$-isomorphic to the full matrix algebra (whose size can depend on $j$), and $F_j :=\{u\in \mathcal{U}(\mathcal{A})\mid f_j\Phi(u)\in \mathcal{U}(f_j\mathcal{B})\}\subset \mathcal{U}(\mathcal{A})$ is nowhere dense, $j\geq 1$. Then the Baire Category Theorem shows that $\mathcal{U}(\mathcal{A})\setminus \bigcup_{j\geq 1} F_j$ is nonempty. This clearly contradicts the preceding lemma. 
Therefore, there is a maximal family of distinct central projections $f_j$, $1\leq j\leq k$, in $\mathcal{B}$ such that each $f_j\mathcal{B}$ is $^*$-isomorphic to the full matrix algebra (whose size can depend on $j$), and $F_j :=\{u\in \mathcal{U}(\mathcal{A})\mid f_j\Phi(u)\in \mathcal{U}(f_j\mathcal{B})\}\subset \mathcal{U}(\mathcal{A})$ is nowhere dense. 
Then, the preceding lemma implies that $q=f_1+f_2+\cdots+f_k$ satisfies the desired property.
Indeed, if $u\in \mathcal{U}(\mathcal{A})$ satisfies $(1-q)\Phi(u)\notin \mathcal{U}((1-q)\mathcal{B})$, then the preceding lemma implies that there is a nonzero petty central projection $f\leq 1-q$ in $\mathcal{B}$ with $f\Phi(u)\notin \mathcal{U}(f\mathcal{B})$. Since $f\mathcal{B}$ is a direct sum of matrix algebras, we get a contradiction with the maximality of the choice of $f_j$, $1\leq j\leq k$.
\end{proof}

\begin{proof}[Proof of Theorem \ref{main}]
Take the projection $q\in \mathcal{B}$ as in Lemma \ref{qb}. 
Since the mapping $\mathcal{A}\ni a\mapsto (1-q)\Phi(a)\in (1-q)\mathcal{B}$ sends unitaries to unitaries, Theorem \ref{russodye} shows that the mapping $\mathcal{A}\ni a\mapsto \Phi(1)^*(1-q)\Phi(a)=(1-q)\Phi(1)^*\Phi(a)\in (1-q)\mathcal{B}$ is a unital Jordan $^*$-homomorphism. 
By Corollary \ref{kernel}, the kernel $\mathcal{K}$ of this Jordan $^*$-homomorphism is a closed ideal in the C$^*$-algebra $\mathcal{A}$. 
On the other hand, since $\Phi $ is injective with finite-codimensional range and $\dim q\mathcal{B}<\infty$, we also see that $\dim \mathcal{K}<\infty$. 
By Lemma \ref{fdideal}, there is a central projection $p\in \mathcal{A}$ such that $\mathcal{K}=p\mathcal{A}$. 
Then the mapping $J\colon (1-p)\mathcal{A}\in a\mapsto (1-q)\Phi(1)^*\Phi(a)\in (1-q)\mathcal{B}$ is an injective unital Jordan $*$-homomorphism, which is an isometry by Lemma \ref{isometry}. 
It follows that $\Phi_2\colon (1-p)\mathcal{A}\ni a\mapsto (1-q)\Phi(a)\in (1-q)\mathcal{B}$ is an isometry.
Moreover, the mapping $\Phi _1\colon p\mathcal{A}\ni a\mapsto \Phi(a)=q\Phi(a)\in q\mathcal{B}$ is an isometry. Finally, set $\Phi _3\colon (1-p)\mathcal{A}\to q\mathcal{B}$ by $\Phi _3(a)=q\Phi(a)$, $a\in (1-p)\mathcal{A}$. 
Then $p, q$ and $\Phi _1, \Phi_2, J, \Phi_3$ together with $v=(1-q)\Phi(1)\in \mathcal{U}((1-q)\mathcal{B})$ satisfy all of the desired properties.
\end{proof}

\subsection{Isometries with corank $1$}
We extract the special case of corank $1$ in the above proofs. 
\begin{lemma}\label{corank1lemma}
Assume that $\Phi\colon \mathcal{A}\to \mathcal{B}$ is a linear isometry with corank $1$, and that there is $u\in \mathcal{U}(\mathcal{A})$ such that $\Phi(u) \notin \mathcal{U}(\mathcal{B})$. Then there is a central atom $q_0\in \mathcal{B}$ such that $(1-q_0)\Phi(\mathcal{U}(\mathcal{A}))\subset \mathcal{U}((1-q_0)\mathcal{B})$.
\end{lemma}
\begin{proof}
Using the symbols and methods as in the proof of Lemma \ref{kadison}, we see that either $p_0=v_0^*v_0\neq 1$ or $v_0v_0^*\neq 1$ holds. 
Assume that $p_0\neq 1$ holds. Since $\Phi_0\colon \mathcal{A} \to p_0\mathcal{B}p_0$ is injective and $\Phi$ is an isometry with corank $1$, we see that $q_0=1-p_0\in \mathcal{B}$ is an atom satisfying $q_0\mathcal{B}(1-q_0)=\{0\}$. Thus $q_0$ is central in $\mathcal{B}$. 
Since $q_0\mathcal{B}$ is a $1$-dimensional C$^*$-algebra, we see that $q_0\neq q_0v_0^*v_0=q_0v_0v_0^*$, so $v_0v_0^*\neq 1$.
Now, imitating the above discussion with $v_0v_0^*$ in place of $p_0$, we see that $1-v_0v_0^*$ also needs to be an atom, and we get to the equation $1-v_0v_0^*=q_0$. Consequently, we get $1-v_0^*v_0=1-v_0v_0^*=q_0$.
We can get to the same conclusion by assuming $v_0v_0^*\neq 1$. 

Looking at the proof of Lemma \ref{qb}, we see that the set $\{u_1\in \mathcal{U}(\mathcal{A})\mid q_0\Phi(u_1)\in \mathcal{U}(q_0\mathcal{B})\}$ is nowhere dense. On the other hand, if the set $\{u_1\in \mathcal{U}(\mathcal{A})\mid (1-q_0)\Phi(u_1)\in \mathcal{U}((1-q_0)\mathcal{B})\}$ is not equal to $\mathcal{U}(\mathcal{A})$, then this set also needs to be nowhere dense. 
Then one may take an element $u_2\in \mathcal{U}(\mathcal{A})$ satisfying $q_0\Phi(u_2)\notin \mathcal{U}(q_0\mathcal{B})$ and $(1-q_0)\Phi(u_2)\notin \mathcal{U}((1-q_0)\mathcal{B})$.
This contradicts the fact that $1-\Phi(u_2)^*\Phi(u_2)=1-\Phi(u_2)\Phi(u_2)^*$ needs to be a scalar multiple of an atom which is proved in the first paragraph.
\end{proof}

Now we can prove Proposition \ref{corank1}.
\begin{proof}[Proof of Proposition \ref{corank1}]
If there is $u\in \mathcal{U}(\mathcal{A})$ such that $\Phi(u)\notin \mathcal{U}(\mathcal{B})$, then Lemma \ref{corank1lemma} implies that there is a central atom $q_0\in \mathcal{B}$ with the property that$(1-q_0)\Phi(\mathcal{U}(\mathcal{A}))\subset \mathcal{U}((1-q_0)\mathcal{B})$.
By the Russo--Dye theorem (Theorem \ref{russodye}), we see that the mapping $J_1\colon \mathcal{A}\ni a\mapsto \Phi(1)^* (1-q_0)\Phi(a)\in (1-q_0)\mathcal{B}$ is a unital Jordan $^*$-homomorphism. 
Since $\Phi$ is an isometry and $q_0$ is a central atom in $\mathcal{B}$, we see that $\dim\ker J_1\leq 1$. 
Assume that $\ker J_1=0$, i.e., $J_1$ is injective. Since $\Phi$ has corank $1$, we see that $J_1$ is surjective. It follows that $J_1\colon \mathcal{A}\to (1-q_0)\mathcal{B}$ is a surjective isometry and hence the mapping $\mathcal{A}\ni a\mapsto (1-q_0)\Phi(a)\in  (1-q_0)\mathcal{B}$ is also a surjective isometry.
Assume $\dim\ker J_1=1$ (we will obtain a contradiction). By Corollary \ref{kernel}, $\ker J_1$ is a $1$-dimensional ideal of $\mathcal{A}$. Thus there is a central atom $e\in \mathcal{A}$ such that $\ker J_1 = e\mathcal{A}$. It follows that $\Phi(e)=q_0\Phi(e)$. Since $\lVert \Phi(e)\rVert =\lVert e\rVert =1$, we get $\Phi(e)=\lambda q_0$ for some $\lambda \in \mathbb{T}$. 
If $a\in (1-e)\mathcal{A}^{\leq 1}$, then we have $e\pm a \in \mathcal{A}^{\leq 1}$. 
It follows that $\mathcal{B}^{\leq 1}\ni \Phi(e\pm a) = q_0(\lambda \pm \Phi(a))  +(1-q_0)\Phi(a)$, so $q_0\Phi(a)=0$. This leads to the fact that $\Phi(\mathcal{U}(\mathcal{A}))\subset \mathcal{U}(\mathcal{B})$, which contradicts the assumption. 
To see this, let $u\in \mathcal{U}(\mathcal{A})$. One may find $\mu\in \mathbb{T}$ and $u'\in \mathcal{U}((1-e)\mathcal{A})$ with $u=\mu e+u'$. Then $\Phi(u) = \mu\lambda q_0 + (1-q_0)\Phi(u')$. Lemma \ref{corank1lemma} implies that $(1-q_0)\Phi(u')\in \mathcal{U}((1-q_0)\mathcal{B})$. This together with $\mu\lambda\in \mathbb{T}$ shows $\Phi(u)\in \mathcal{U}(\mathcal{B})$.

If $\Phi(\mathcal{U}(\mathcal{A}))\subset \mathcal{U}(\mathcal{B})$, then the Russo--Dye theorem (Theorem \ref{russodye}) implies that the mapping $\mathcal{A}\ni a\mapsto \Phi(1)^*\Phi(a)\in \mathcal{B}$ is an injective unital Jordan $^*$-homomorphism with corank $1$. 
\end{proof}

\section{Self-adjoint parts}\label{S3}
Now we turn our attention to the self-adjoint part. 
Some results can be obtained in parallel with the discussion as in Section \ref{secmain}, but many others are not. 

\subsection{Isometries with finite corank}
\begin{lemma}\label{e+-}
Let $\Phi \colon \mathcal{A}_{sa}\to \mathcal{B}_{sa}$ be a linear isometry with finite corank. 
Set $f_+:=\chi_{\{1\}}(\Phi(1))$ and $f_-:=\chi_{\{-1\}}(\Phi(1))$. 
Then the mapping 
\[
\Psi\colon \mathcal{A}_{sa}\ni a\mapsto f_+ \Phi(a)f_+-f_-\Phi(a)f_-\in f_+\mathcal{B}_{sa}f_+ +f_-\mathcal{B}_{sa}f_-
\]
is a unital isometry with finite corank.
\end{lemma}
\begin{proof}
Since $1\in \mathrm{ext}\,A_{sa}^{\leq 1}$ and $\Phi$ is isometric, we have $\Phi(1)\in \mathrm{ext}\, \Phi(\mathcal{A}_{sa})^{\leq 1}$. 
By Proposition \ref{ppp}, we see that $f_\pm$ makes sense and that $1-f_+-f_-$ is a sum of finitely many atoms in $\mathcal{B}_{sa}$. 
It is clear that $\Psi$ is unital, with norm at most $1$, and has finite corank.
From here we will give a self-adjoint variant of the discussion as in the first paragraph of the proof of Lemma \ref{kadison}. 
Let $a\in \mathcal{A}_{sa}$ be an element satisfying $\lVert a\rVert=1=\max\sigma(a)$. 
Then we have 
\[
2=\lVert 1+a\rVert =\lVert \Phi(1+a)\rVert =\lVert \Phi(1)+\Phi(a)\rVert.
\] 
We may assume that $\mathcal{B}$ acts on a Hilbert space $\mathcal{H}$. Then we may take a sequence of unit vectors $h_n\in \mathcal{H}$ satisfying either 
$\langle (\Phi(1)+\Phi(a))h_n, h_n\rangle\to \lVert \Phi(1)+\Phi(a)\rVert =2$ or $\langle (\Phi(1)+\Phi(a))h_n, h_n\rangle\to -\lVert \Phi(1)+\Phi(a)\rVert =-2$ as $n\to\infty$. 
Assume that the former condition holds.
Since $\lVert \Phi(1)\rVert =\lVert \Phi(a)\rVert =1$, we get $\langle \Phi(1)h_n, h_n\rangle \to 1$ and  $\langle \Phi(a)h_n, h_n\rangle \to 1$. 
Then, considering the spectral decomposition of $\Phi(1)$, one may see that $\lVert f_+h_n-h_n\rVert \to 0$ as $n\to \infty$. It follows that $\langle f_+\Phi(a)f_+h_n, h_n\rangle=\langle \Phi(a)f_+h_n, f_+h_n\rangle \to 1$, so we get $\lVert f_+\Phi(a)f_+\rVert =1$. 
Similarly, if the latter condition holds, one may see that $\lVert f_-\Phi(a)f_-\rVert =1$. 
Consequently, we get $\lVert \Psi(a)\rVert=\lVert a\rVert$. 
Since $\Psi$ is linear, the same equality holds for every $a\in \mathcal{A}_{sa}$, so $\Psi$ is a linear isometry. 
\end{proof}

\begin{lemma}\label{f123}
Let $\Psi\colon \mathcal{A}_{sa}\to \mathcal{B}_{sa}$ be a unital linear isometry with finite corank. 
Let $e_1, e_2, \ldots, e_l$ be mutually orthogonal nonzero projections in $\mathcal{A}$.
For each $j$, set $f_j:=\chi_{\{1\}}(\Psi(e_j))$. Then $f_1, f_2, \ldots, f_l$ are mutually orthogonal. Moreover, the mapping $e_j \mathcal{A}_{sa}e_j\ni a \mapsto f_j\Psi(a)f_j\in f_j\mathcal{B}_{sa}f_j$ is a unital linear isometry.
\end{lemma}
\begin{proof}
For each $j$, we have $0\leq \Psi(e_j)\leq 1$.
Since $2e_j-1\in \mathrm{ext}\, \mathcal{A}_{sa}^{\leq1}$, we get $2\Psi(e_j)-1=\Psi(2e_j-1)\in \mathrm{ext}\, \Psi(\mathcal{A}_{sa})^{\leq 1}$. 
Thus Proposition \ref{ppp} implies that $\#\sigma(\Psi(e_j))<\infty$ and that $\chi_{(0,1)}(\Psi(e_j))$ is a linear combination of finitely many atoms. 
Since $e_1, e_2, \ldots, e_l$ are mutually orthogonal, we have
\[
1= \lVert e_1+e_2+\cdots + e_l \rVert = \lVert \Psi(e_1)+\Psi(e_2)+\cdots + \Psi(e_l) \rVert.
\]
This together with 
\[
\Psi(e_1)+\Psi(e_2)+\cdots + \Psi(e_l)\geq f_1+f_2+\cdots + f_l\geq 0 
\]
implies that $\lVert f_1+f_2+\cdots + f_l\rVert \leq 1$. Since $f_1, f_2, \ldots, f_l$ are projections, we see that they are mutually orthogonal. 

Using the fact that $\Psi$ is positive, one may show that the mapping $e_j \mathcal{A}_{sa}e_j\ni a \mapsto f_j\Psi(a)f_j\in f_j\mathcal{B}_{sa}f_j$ is an isometry by imitating the discussion as in the proof of the preceding lemma.
\end{proof}

A special case of the preceding lemma follows.
\begin{lemma}\label{f1234}
Let $\Psi\colon \mathcal{A}_{sa}\to \mathcal{B}_{sa}$ be a unital linear isometry with finite corank. 
Let $e_1, e_2, \ldots, e_l$ be mutually orthogonal nonzero central projections in $\mathcal{A}$ with $e_1+e_2+\cdots+e_l=1$.
For each $j$, set $f_j:=\chi_{\{1\}}(\Psi(e_j))$. Then the mapping 
\[
\mathcal{A}_{sa}\ni a \mapsto f_1\Psi(a)f_1+ f_2\Psi(a)f_2+\cdots + f_l \Psi(a) f_l\in f_1\mathcal{B}_{sa}f_1+f_2\mathcal{B}_{sa}f_2+\cdots +f_l\mathcal{B}_{sa}f_l
\]
is a unital linear isometry.
\end{lemma}
\begin{proof}
Let $a\in \mathcal{A}_{sa}$.
By the assumption, we have $a=e_1a+e_2a+\cdots +e_la$, and each $e_ja$ belongs to $e_j\mathcal{A}_{sa}e_j$.
Let us show that $f_j\Psi(a)f_j=f_j\Psi(e_ja)f_j$. It suffices to verify $f_j\Psi((1-e_j)a)f_j=0$.
This follows from $-\lVert a\rVert \Psi(1-e_j)\leq \Psi((1-e_j)a)\leq \lVert a\rVert \Psi(1-e_j)$, $\Psi(1-e_j)=1-\Psi(e_j)$, and $f_j=\chi_{\{1\}}(\Psi(e_j))$.
Now the preceding lemma completes the proof.
\end{proof}

\begin{lemma}\label{e123}
Let $e_1, e_2, e_3\in \mathcal{A}$ be mutually orthogonal projections and $z\in \mathcal{A}$. Let $c$ be a real number with $-1\leq c\leq1$. If $e_2ze_3=z\neq 0$, then we have 
\[
\lVert -e_1+ce_2+e_3+z+z^*\rVert=\max\sigma(-e_1+ce_2+e_3+z+z^*) =\frac{1+c}{2}+\sqrt{\left(\frac{1-c}{2}\right)^2+\lVert z\rVert^2}.
\]
\end{lemma}
\begin{proof}
Since $e_1, e_2, e_3$ are mutually orthogonal and
\[
\frac{1+c}{2}+\sqrt{\left(\frac{1-c}{2}\right)^2+\lVert z\rVert^2}> 1,
\]
it suffices to show 
\[
\lVert ce_2+e_3+z+z^*\rVert=\max\sigma(ce_2+e_3+z+z^*) =\frac{1+c}{2}+\sqrt{\left(\frac{1-c}{2}\right)^2+\lVert z\rVert^2}.
\]
To show it, we only need to verify 
\begin{equation}\label{eq1}
\lVert ce_2+e_3+z+z^*\rVert\leq \frac{1+c}{2}+\sqrt{\left(\frac{1-c}{2}\right)^2+\lVert z\rVert^2}
\end{equation}
and 
\begin{equation}\label{eq2}
\frac{1+c}{2}+\sqrt{\left(\frac{1-c}{2}\right)^2+\lVert z\rVert^2}\in \sigma(ce_2+e_3+z+z^*).
\end{equation}

Let $c_2, c_3$ be positive real numbers. Then $c_2e_2+c_3e_3\pm(z+z^*)\geq 0$ holds if and only if $\langle (c_2e_2+c_3e_3\pm(z+z^*))(g+h), g+h\rangle\geq 0$ for every $g$ in the range of $e_2$ and every $h$ in the range of $e_3$. 
By
\[
\langle (c_2e_2+c_3e_3\pm(z+z^*))(g+h), g+h\rangle = c_2\langle g, g\rangle +  c_3\langle h, h\rangle \pm 2\mathrm{Re}\, \langle zh, g\rangle 
\]
and the Cauchy--Schwarz inequality, we get to the fact that $c_2e_2+c_3e_3\pm(z+z^*)\geq 0$ holds if and only if $\lVert z\rVert^2\leq c_2c_3$. 

Therefore, for a real number $t>1$, we have 
\[
\begin{split}
ce_2+e_3+z+z^*\leq t(e_2+e_3) &\iff (t-c)e_2+(t-1)e_3-z-z^*\geq 0\\
&\iff (t-c)(t-1)\geq \lVert z\rVert^2\\
&\iff t\geq \left(\frac{1+c}{2}+\sqrt{\left(\frac{1-c}{2}\right)^2+\lVert z\rVert^2}\right).
\end{split}
\]
In particular, we obtain \eqref{eq2} and 
\begin{equation}\label{eq3}
ce_2+e_3+z+z^*\leq \left(\frac{1+c}{2}+\sqrt{\left(\frac{1-c}{2}\right)^2+\lVert z\rVert^2}\right)(e_2+e_3).
\end{equation}
We also have
\[
\left(\frac{-1+c}{2}+\sqrt{\left(\frac{1-c}{2}\right)^2+\lVert z\rVert^2}\right)e_2+\left(\frac{1-c}{2}+\sqrt{\left(\frac{1-c}{2}\right)^2+\lVert z\rVert^2}\right)e_3 +z+z^*\geq 0,
\]
which implies
\begin{equation}\label{eq4}
\left(\frac{1+c}{2}-\sqrt{\left(\frac{1-c}{2}\right)^2+\lVert z\rVert^2}\right)(e_2+e_3)\leq ce_2+e_3+z+z^*.
\end{equation}
Therefore, using \eqref{eq3}, \eqref{eq4} together with the assumption $c\geq-1$, we obtain \eqref{eq1}.
\end{proof}

\begin{lemma}\label{1-chi}
Let $\Phi\colon \mathcal{A}_{sa}\to \mathcal{B}_{sa}$ be a linear isometry with finite corank. 
Let $s$ be a self-adjoint unitary in $\mathcal{A}$. 
Then $\chi_{(-1, 1)}(\Phi(s))\in \mathcal{B}$ is petty.
\end{lemma}
\begin{proof}
Set $f_{1}=\chi_{\{1\}}(\Phi(s))$, $f_{-1}=\chi_{\{-1\}}(\Phi(s))$,  and $f=\chi_{(-1,1)}(\Phi(s)) =1-f_1-f_{-1}$.
By Proposition \ref{ppp}, we know that $f$ is a sum of finitely many atoms, so $\dim f\mathcal{B}f<\infty$. 
Assume that $f\in \mathcal{B}$ is not petty. By Lemma \ref{epetty}, we have either $\dim f\mathcal{B}f_1  =\infty$ or $\dim f\mathcal{B} f_{-1}=\infty$. 
Assume that the former holds. Then $\{b+b^*\mid b\in f\mathcal{B} f_1\} \subset \mathcal{B}_{sa}$ is also infinite-dimensional.
Since $\Phi$ has finite corank, there is a nonzero element $b\in f\mathcal{B} f_1\neq 0$ such that $b+b^*\in \Phi (\mathcal{A}_{sa})$.
Take $a\in \mathcal{A}_{sa}$ satisfying $\Phi(a)=b+b^*$. 
Then, for each $t\in \mathbb{R}$, we have 
\[
\lVert s+ta\rVert =\lVert \Phi(s+ta)\rVert = \lVert \Phi(s)+t(b+b^*)\rVert. 
\]

Now, for $t\in \mathbb{R}$, we estimate the norm $\lVert \Phi(s)+t(b+b^*)\rVert$. 
Take a positive real number $0<\varepsilon<1$ satisfying $(-1,-1+\varepsilon)\cap \sigma(\Phi(s))=(1-\varepsilon, 1)\cap \sigma(\Phi(s))=\emptyset$. Then we have
\[
x+t(b+b^*)\leq \Phi(s)+t(b+b^*)\leq y+t(b+b^*),
\]
where $x=-f_{-1} + (-1+\varepsilon)f +f_1$ and $y=-f_{-1} + (1-\varepsilon)f +f_1$.
An application of the preceding lemma shows that 
\[
\lVert x+t(b+b^*)\rVert = \frac{\varepsilon}{2}+\sqrt{\left(1-\frac{\varepsilon}{2}\right)^2 +t^2\lVert b\rVert^2}
\]
and 
\[
\lVert y+t(b+b^*)\rVert = 1-\frac{\varepsilon}{2}+\sqrt{\frac{\varepsilon^2}{4} +t^2\lVert b\rVert^2}.
\] 
Since $0<\varepsilon<1$, we obtain 
\[
\lVert \Phi(s)+t(b+b^*)\rVert \leq  \max\{\lVert x+t(b+b^*)\rVert, \lVert y+t(b+b^*)\rVert\} =  1-\frac{\varepsilon}{2}+\sqrt{\frac{\varepsilon^2}{4} +t^2\lVert b\rVert^2}.
\]

We are going to show that $eae=0=(1-e)a(1-e)$, where $e=(s+1)/2$.
Assume that $\lambda\in \sigma(eae)$. For each $t\in \mathbb{R}$, we have $1+t\lambda\in\sigma(e+teae)$ and hence $\lVert e+teae\rVert\geq 1+t\lambda$. Thus we obtain $\lVert s+ta\rVert\geq \lVert e(s+ta)e\rVert= \lVert e+teae\rVert \geq 1+t\lambda$.
Therefore, we get
\[
1+t\lambda\leq \lVert s+ta\rVert=\lVert \Phi(s)+t(b+b^*)\rVert\leq 1-\frac{\varepsilon}{2}+\sqrt{\frac{\varepsilon^2}{4} +t^2\lVert b\rVert^2}
\]
and hence
\[
\frac{\varepsilon}{2}+t\lambda\leq \sqrt{\frac{\varepsilon^2}{4} +t^2\lVert b\rVert^2}.
\]
Since this holds for every $t\in \mathbb{R}$, we get $\lambda=0$ and thus $\sigma(eae)=\{0\}$ or equivalently, $eae=0$. 
A similar argument shows $(1-e)a(1-e)=0$ as well.

It follows from Lemma \ref{e123} (with $e_1=0$ and $c=-1$) that $\lVert s+ta\rVert= \sqrt{1+t^2\lVert a\rVert^2}$.
Lemma \ref{e123} also implies that 
\[
\frac{\varepsilon}{2}+\sqrt{\left(1-\frac{\varepsilon}{2}\right)^2 +t^2\lVert b\rVert^2}\in\sigma (x+t(b+b^*)).
\] 
This together with $x+t(b+b^*)\leq \Phi(s)+t(b+b^*)$ gives
\[
\frac{\varepsilon}{2}+\sqrt{\left(1-\frac{\varepsilon}{2}\right)^2 +t^2\lVert b\rVert^2}\leq \max\sigma(\Phi(s)+t(b+b^*)).
\]
Since the left-hand side is a positive real number, we obtain
\[
\frac{\varepsilon}{2}+\sqrt{\left(1-\frac{\varepsilon}{2}\right)^2 +t^2\lVert b\rVert^2}\leq \lVert \Phi(s)+t(b+b^*)\rVert.
\]
Thus we get 
\[
\frac{\varepsilon}{2}+\sqrt{\left(1-\frac{\varepsilon}{2}\right)^2 +t^2\lVert b\rVert^2}\leq \lVert \Phi(s)+t(b+b^*)\rVert=\lVert s+ta\rVert=\sqrt{1+t^2\lVert a\rVert^2}.
\]
This together with the equality $\lVert a\rVert =\lVert \Phi(a)\rVert =\lVert b+b^*\rVert =\lVert b\rVert$ leads to a contradiction. 
Similarly, we get to a contradiction if $\dim f\mathcal{B}f_{-1}=\infty$. 
Thus $f=\chi_{(-1, 1)}(\Phi(s))\in \mathcal{B}$ is petty.
\end{proof}

\begin{corollary}\label{01chi}
Let $\Psi\colon \mathcal{A}_{sa}\to \mathcal{B}_{sa}$ be a unital linear isometry with finite corank. 
If $e\in \mathcal{P}(\mathcal{A})$, then $\chi_{(0, 1)}(\Psi(e))\in \mathcal{B}$ is petty.
\end{corollary}

Now, in order to make use of the full power of projections, we study the case of von Neumann algebras. 
Later we will study the original problem for C$^*$-algebras using the bidual mapping. 
In what follows, we assume that $\mathcal{M}$, $\mathcal{N}$ are von Neumann algebras.
The following two lemmas give general facts about von Neumann algebras.
\begin{lemma}\label{generate}
Assume that $\mathcal{M}$ admits no direct summand of type I$_1$.
Then there is a projection $e_1\in \mathcal{P}(\mathcal{M})$ with the following property: For any neighborhood $G$ of $e_1$ in $\mathcal{M}_{sa}$, the Jordan subalgebra of $\mathcal{M}_{sa}$ generated by $G\cap\mathcal{P}(\mathcal{M})$ coincides with $\mathcal{M}_{sa}$.   
\end{lemma}
\begin{proof}
By considering the type decomposition of $\mathcal{M}$, we see that we may take mutually orthogonal projections $e_1, e_2\in \mathcal{P}(\mathcal{M})$ satisfying $1-e_1-e_2\prec e_1\sim e_2$.  
Let $G$ be a neighborhood of $e_1$ in $\mathcal{M}_{sa}$, and let $\mathcal{M}_1$ be the Jordan subalgebra of $\mathcal{M}_{sa}$ generated by $G\cap\mathcal{P}(\mathcal{M})$.

Let $v\in \mathcal{M}$ be a partial isometry satisfying $v^*v\leq e_1$ and $vv^*\leq 1-e_1$.
Then, for every positive operator $a\in v^*v\mathcal{M}_{sa}v^*v$ with sufficiently small norm, we have 
\[
e_{\pm}=(e_1-a) \pm va^{1/2}(1-a)^{1/2}  \pm a^{1/2}(1-a)^{1/2}v^*+ vav^*\in G\cap\mathcal{P}(\mathcal{M}).
\]
Thus we obtain 
\[
\frac{e_++e_-}{2}=(e_1-a)+vav^*, \quad\frac{e_+-e_-}{2}=va^{1/2}(1-a)^{1/2}  + a^{1/2}(1-a)^{1/2}v^*\in \mathcal{M}_1.
\]
We also have 
\[
\frac{e_1((e_1-a)+vav^*)+((e_1-a)+vav^*)e_1}{2}=e_1-a\in \mathcal{M}_1
\]
and hence
\[
e_1-(e_1-a)=a,\quad ((e_1-a)+vav^*)-(e_1-a)=vav^*\in \mathcal{M}_1.
\] 
Considering the linear span of operators that can be obtained in the above way for various partial isometries $v$ satisfying $v^*v\leq e_1$ and $vv^*\leq 1-e_1$, and various positive operators $a\in v^*v\mathcal{M}_{sa}v^*v$ with sufficiently small norm, we reach the desired conclusion $\mathcal{M}_{sa}\subset \mathcal{M}_1$.
\end{proof}

\begin{lemma}\label{e1e2}
Let $\mathcal{P}\subset \mathcal{P}(\mathcal{M})$ be a connected component. 
If $e_1, e_2\in \mathcal{P}$, then there is a unitary $u\in \mathcal{M}$ satisfying $ue_1u^*=e_2$.
\end{lemma}
\begin{proof}
Since this is a well-known fact, we only give a sketch of the proof.
It suffices to consider the case where $\lVert e_1-e_2\rVert$ is sufficiently small. 
In that case, let $e_2e_1=v\lvert e_2e_1\rvert$ be the polar decomposition of $e_2e_1$. 
Then it can be seen that $v\in \mathcal{M}$ is a partial isometry satisfying $v^*v=e_1$ and $vv^*=e_2$. 
Since $\lVert (1-e_1)-(1-e_2)\rVert=\lVert e_1-e_2\rVert$ is small, we may also take a partial isometry $w\in \mathcal{M}$ satisfying $w^*w=1-e_1$ and $ww^*=1-e_2$. Then $u=v+w$ is a unitary satisfying the desired property. 
\end{proof}

\begin{lemma}\label{abe}
Let $\Psi\colon \mathcal{M}_{sa}\to \mathcal{N}_{sa}$ be a unital linear isometry with finite corank. 
Then there are a petty central projection $f\in \mathcal{N}$ and a Jordan $^*$-homomorphism $J\colon \mathcal{M}\to (1-f)\mathcal{N}$ satisfying $(1-f)\Psi(a)=J(a)$ for every $a\in \mathcal{M}_{sa}$.  
\end{lemma}
\begin{proof}
The proof is separated into three parts. \medskip

\noindent \underline{Case 1} We first consider the case where $\mathcal{M}$ admits no direct summand of type I$_1$. 
Take a projection $e_1\in \mathcal{M}$ as in Lemma \ref{generate}. 
Let $\mathcal{P}$ denote the connected component of $\mathcal{P}(\mathcal{M})$ with $e_1\in \mathcal{P}$. 
By Lemma \ref{e1e2}, every element in $\mathcal{P}$ has the property of $e_1$ as in Lemma \ref{generate}.
From now on, our discussion proceeds in parallel with that in the proof of Lemma \ref{qb}.

Let $f$ be a central projection in $\mathcal{N}$ such that $f\mathcal{N}$ is $^*$-isomorphic to the full matrix algebra $\mathbb{M}_m$ for some $m$.
If $\{e\in \mathcal{P}\mid f\Psi(e)\in \mathcal{P}(f\mathcal{N})\}$ has nonempty interior as a subset of $\mathcal{P}$, then it is contained in $\mathcal{M}_1:=\{a\in \mathcal{M}_{sa}\mid f\Psi(a)^2=f\Psi(a^2)\}$. 
On the other hand, since the mapping $\mathcal{M}_{sa}\ni a\mapsto f\Psi(a)\in f\mathcal{N}_{sa}$ is a positive linear map with norm at most $1$, the set $\mathcal{M}_1\subset \mathcal{M}_{sa}$ is a Jordan subalgebra of $\mathcal{M}_{sa}$ by Lemma \ref{broise}. Therefore, we get $\mathcal{M}_1=\mathcal{M}_{sa}$. It follows that $f\Psi(e)\in \mathcal{P}(f\mathcal{N})$ for every $e\in \mathcal{P}$. 
Consequently, the closed subset $\{e\in \mathcal{P}\mid f\Psi(e)\in \mathcal{P}(f\mathcal{N})\}\subset \mathcal{P}$ is either nowhere dense or equal to $\mathcal{P}$. 

Assume that there is a sequence of distinct central projections $f_j$, $j\geq 1$, in $\mathcal{N}$ such that each $f_j\mathcal{N}$ is $^*$-isomorphic to the full matrix algebra (whose size can depend on $j$), and $F_j :=\{e\in \mathcal{P}\mid f_j\Psi(e)\in \mathcal{P}(f_j\mathcal{N})\}\subset \mathcal{P}$ is nowhere dense, $j\geq 1$. Then the Baire Category Theorem shows that $\mathcal{P}\setminus \bigcup_{j\geq 1} F_j$ is nonempty. Taking a projection of this set, we get from Corollary \ref{01chi} a contradiction.
Therefore, there is a maximal family of distinct central projections $f_j$, $1\leq j\leq k$, in $\mathcal{N}$ such that each $f_j\mathcal{N}$ is $^*$-isomorphic to the full matrix algebra (whose size can depend on $j$), and $F_j :=\{e\in \mathcal{P}\mid f_j\Psi(e)\in \mathcal{P}(f_j\mathcal{N})\}\subset \mathcal{P}$ is nowhere dense. 
Then, the projection $f=f_1+f_2+\cdots+f_k$ satisfies the desired property.
Indeed, Corollary \ref{01chi} implies that $(1-f)\Psi(e)\in \mathcal{P}((1-f)\mathcal{N})$ holds for every $e\in \mathcal{P}$. 
Since $\mathcal{P}$ generates $\mathcal{M}_{sa}$ as a Jordan algebra and $\Psi$ is a unital positive linear mapping, Lemma \ref{broise} implies that the mapping $\mathcal{M}_{sa}\ni a\mapsto (1-f)\Psi(a)\in (1-f)\mathcal{N}_{sa}$ extends to a Jordan $^*$-homomorphism from $\mathcal{M}$ to $(1-f)\mathcal{N}$.\medskip

\noindent \underline{Case 2} We next work on the case where $\mathcal{M}$ is abelian.
Let $d$ be the corank of $\Psi$. 
Let $e_1, e_2, \cdots, e_m$ be mutually orthogonal projections in $\mathcal{M}$ whose sum is $1$. 
Set $f_j=\chi_{\{1\}}(\Psi(e_j))$. Then 
it follows from Lemma \ref{f1234} that the mapping 
\[
\mathcal{M}_{sa}\ni a\mapsto f_1 \Psi(a)f_1 + f_2\Psi(a)f_2+\cdots + f_m\Psi(a)f_m\in \mathcal{N}_{sa}
\]
is an isometry. Therefore, the mapping 
\[
\mathcal{M}_{sa}\ni a\mapsto (f_1+f_2+\cdots+f_m) \Psi(a)(f_1 + f_2+\cdots+f_m)\in \mathcal{N}_{sa}
\]
is also an isometry.
Since the corank of $\Psi$ is $d$, we see that the dimension of the real linear space $\{y\in \mathcal{N}_{sa}\mid (f_1+f_2+\cdots+f_m) y(f_1 + f_2+\cdots+f_m)=0\}$ is at most $d\,(<\infty)$. 
Therefore, we may retake a family $e_1, e_2, \cdots, e_m$ of mutually orthogonal central projections in $\mathcal{M}$ ($m$ can be different from the above choice) whose sum is $1$ in such a way that the dimension of the space $\{y\in \mathcal{N}_{sa}\mid (f_1+f_2+\cdots+f_m) y(f_1 + f_2+\cdots+f_m)=0\}$ is maximal. By Lemma \ref{eepetty}, this space is petty as a subset of $\mathcal{N}$.
Take the central projection $f\in \mathcal{N}$ such that the ideal of $\mathcal{N}$ generated by $1-(f_1+f_2+\cdots+f_m)$ is equal to  $f\mathcal{N}$.
Then we have $(1-f)\Psi(e)\in \mathcal{P}((1-f)\mathcal{N})$ for every $e\in \mathcal{P}(\mathcal{M})$.
Indeed, if $e\in \mathcal{P}(\mathcal{M})$ satisfies $(1-f)\Psi(e)\notin \mathcal{P}((1-f)\mathcal{N})$, then it is not hard to see that the family $ee_1, ee_2, \cdots, ee_m, (1-e)e_1, (1-e)e_2, \cdots, (1-e)e_m$ gives a contradiction to the maximality of the dimension in the above discussion.
(To see this, for each $j\in \{1,2,\ldots, m\}$, observe that $\Psi(ee_j)$ and $\Psi((1-e)e_{j})$ are positive operators whose sum equals $\Psi(e_j) \, (\leq 1)$. Thus we get  $\chi_{\{1\}}(\Psi(ee_j))+\chi_{\{1\}}(\Psi((1-e)e_j))\leq \chi_{\{1\}}(\Psi(e_j))=f_j$. 
On the other hand, if $\chi_{\{1\}}(\Psi(ee_j))+\chi_{\{1\}}(\Psi((1-e)e_j))= f_j$ for all $j\in \{1,2,\ldots, m\}$, then the sum of $2m$ projections $(1-f)\chi_{\{1\}}(\Psi(ee_j))$ and $(1-f)\chi_{\{1\}}(\Psi((1-e)e_j))$, $j\in \{1,2,\ldots, m\}$, needs to be equal to $1-f$. 
%Thus these projections are orthogonal, and we get . 
Then, by the inequalities $(1-f)\chi_{\{1\}}(\Psi(ee_j))\leq (1-f)\chi_{\{1\}}(\Psi(e))\leq (1-f)\Psi(e)$ and $(1-f)\chi_{\{1\}}(\Psi((1-e)e_j))\leq (1-f)\chi_{\{1\}}(\Psi(1-e))\leq (1-f)\Psi(1-e)$ for each $j$; and the equation $(1-f)\Psi(e)+(1-f)\Psi(1-e)= 1-f$, we see that the sum of $m$ projections $(1-f)\chi_{\{1\}}(\Psi(ee_j))$, $j\in \{1,2,\ldots, m\}$, equals $(1-f)\chi_{\{1\}}(\Psi(e))=(1-f)\Psi(e)$ and it is a projection. 
Therefore, $(1-f)\Psi(e)\notin \mathcal{P}((1-f)\mathcal{N})$ implies that the sum of $2m$ projections $\chi_{\{1\}}(\Psi(ee_j))$, $\chi_{\{1\}}(\Psi((1-e)e_j))$, $j\in \{1,2,\ldots, m\}$ (which is bounded by $f_1+f_2+\cdots+f_m$), is not equal to $f_1+f_2+\cdots+f_m$, contradicting the maximality.)
Consequently, the mapping $\mathcal{M}_{sa}\ni a\mapsto (1-f)\Psi(a)\in (1-f)\mathcal{N}_{sa}$ extends to a Jordan $^*$-homomorphism from $\mathcal{M}$ to $(1-f)\mathcal{N}$.\medskip

\noindent \underline{Case 3} Lastly, we study the general case.
Recall that by the type decomposition we may take a central projection $e\in \mathcal{M}$ such that $e\mathcal{M}$ is abelian and $(1-e)\mathcal{M}$ admits no direct summand of type I$_1$. 
Set $f_1=\chi_{\{1\}}(\Psi(e))$, $f_2=\chi_{\{1\}}(\Psi(1-e))$, and $\mathcal{N}_1=f_1\mathcal{N}f_1$, $\mathcal{N}_2=f_2\mathcal{N}f_2$. By Lemma \ref{f123}, we see that the two mappings $\Psi_1\colon e\mathcal{M}_{sa}\ni a\mapsto f_1\Psi(a)f_1\in (\mathcal{N}_1)_{sa}$ and $\Psi_2\colon (1-e)\mathcal{M}_{sa}\ni a\mapsto f_2\Psi(a)f_2\in (\mathcal{N}_2)_{sa}$ are unital isometries. 
The positivity of $\Psi$ implies that $f_2\Psi(e\mathcal{M}_{sa})=0$ and $f_1\Psi((1-e)\mathcal{M}_{sa})=0$, thus  we see that $\Psi_1$ and $\Psi_2$ are of finite corank.
Therefore, by what we have already shown, there are petty central projections $f_3\in \mathcal{N}_1$, $f_4\in \mathcal{N}_2$ and Jordan $^*$-homomorphisms $J_1\colon e\mathcal{M}\to (f_1-f_3)\mathcal{N}_1$, $J_2\colon (1-e)\mathcal{M}\to (f_2-f_4)\mathcal{N}_2$ such that $J_1(a)=(f_1-f_3)\Psi_1(a)$ for every $a\in e\mathcal{M}_{sa}$ and $J_2(a)=(f_2-f_4)\Psi_2(a)$ for every $a\in (1-e)\mathcal{M}_{sa}$. Define the Jordan $^*$-homomorphism $J_3\colon \mathcal{M}\to \mathcal{N}$ by $J_3(a)=J_1(ea) +J_2((1-e)a)$, $a\in \mathcal{M}$. Then the kernel of $J_3$ is finite-dimensional. For $a\in \mathcal{M}_{sa}$, we have
\[
\begin{split}
J_3(a)&=(f_1-f_3)\Psi_1(ea)+(f_2-f_4)\Psi_2((1-e)a) \\
&=(f_1-f_3)\Psi(ea)(f_1-f_3) + (f_2-f_4)\Psi((1-e)a)(f_2-f_4),
\end{split}
\]
which equals $(f_1-f_3)\Psi(a)(f_1-f_3) + (f_2-f_4)\Psi(a)(f_2-f_4)$ by Lemma \ref{f1234}.
Since $\Psi$ has finite corank and $J_3$ has finite-dimensional kernel, we see that the linear space $\mathcal{V}=\{b\in \mathcal{N}_{sa}\mid (f_1-f_3)b(f_1-f_3) = (f_2-f_4)b(f_2-f_4)=0\}\subset \mathcal{N}_{sa}$ is finite-dimensional. By Lemma \ref{eepetty}, $\mathcal{V}\subset \mathcal{N}$ is petty.
Now, take the central projection $f\in \mathcal{N}$ such that the ideal generated by $\mathcal{V}$ equals $f\mathcal{N}$. 
Then we have $(1-f)\Psi(a) = (1-f)J_3(a)$ for every $a\in \mathcal{M}_{sa}$. It follows that the mapping $\mathcal{M}_{sa}\ni a\mapsto (1-f)\Psi(a)\in (1-f)\mathcal{N}_{sa}$ extends to a Jordan $^*$-homomorphism from $\mathcal{M}$ to $(1-f)\mathcal{N}$.
\end{proof}

Now we begin the proof of Theorem \ref{mainsa}.
We first give a proof for von Neumann algebras. 

\begin{proof}[Proof of Theorem \ref{mainsa} for von Neumann algebras]
Let $\Phi\colon \mathcal{M}_{sa}\to\mathcal{N}_{sa}$ be a linear isometry with finite corank. 
Set $f_+:=\chi_{\{1\}}(\Phi(1))$ and $f_-:=\chi_{\{-1\}}(\Phi(1))$. 
Then the mapping 
\[
\Psi\colon \mathcal{M}_{sa}\ni a\mapsto f_+ \Phi(a)f_+-f_-\Phi(a)f_-\in (\mathcal{N}_1)_{sa}
\]
is a unital isometry with finite corank, where $\mathcal{N}_1= f_+\mathcal{N}f_+ +f_-\mathcal{N}f_-$, by Lemma \ref{e+-}.

By the preceding lemma, we see that there is a petty central projection $f\in \mathcal{N}_1$ such that the mapping $\mathcal{M}_{sa}\ni a\mapsto (1-f)\Psi(a)\in (1-f)(\mathcal{N}_1)_{sa}$ extends to a Jordan $^*$-homomorphism $J_1\colon \mathcal{M}\to(1-f)\mathcal{N}_1$. 
Since $\Psi$ has finite corank and $J_1$ has finite-dimensional kernel, we see that the linear space $\mathcal{V}:=\{b\in \mathcal{N}_{sa}\mid f_1bf_1=f_2bf_2=0\}$ is finite-dimensional, where $f_1=(1-f)f_+$ and $f_2=(1-f)f_-$. By Lemma \ref{eepetty}, we see that $\mathcal{V}\subset \mathcal{N}$ is petty.
Now, take the central projection $q\in \mathcal{N}$ such that the ideal generated by $\mathcal{V}$ equals $q\mathcal{N}$. 
Then, Lemma \ref{eepetty} implies that the projections $g_+=(1-q)f_1$ and $g_-=(1-q)f_2$ are central in $\mathcal{N}$, and we have $g_+\Phi(a)-g_-\Phi(a) = (1-q)J_1(a)$ for every $a\in \mathcal{M}_{sa}$. 
Since the unital Jordan $^*$-homomorphism $J_2\colon \mathcal{M}\ni a \mapsto (1-q)J_1(a)\in (1-q)\mathcal{N}$ has finite-dimensional kernel,
there is a central projection $p\in\mathcal{M}$ such that $\ker J_2=p\mathcal{M}$.

Then the mapping $(1-p)\mathcal{M}_{sa}\ni a\mapsto g_+\Phi(a)-g_-\Phi(a)\in (1-q)\mathcal{N}_{sa}$ extends to an injective unital Jordan $^*$-homomorphism $J\colon (1-p)\mathcal{M}\to (1-q)\mathcal{N}$, which is an isometry. 
Thus the mapping $\Phi_2\colon (1-p)\mathcal{M}_{sa}\ni a\mapsto (1-q)\Phi(a)= (g_+-g_-)J(a) \in (1-q)\mathcal{N}_{sa}$ is also an isometry.
We also see that the mapping $\Phi _1\colon p\mathcal{M}_{sa}\ni a\mapsto \Phi(a)=q\Phi(a)\in q\mathcal{N}_{sa}$ is an isometry. Finally, set $\Phi _3\colon (1-p)\mathcal{M}_{sa}\to q\mathcal{N}_{sa}$ by $\Phi _3(a)=q\Phi(a)$, $a\in (1-p)\mathcal{M}_{sa}$. 
Then $p, q$ and $\Phi _1, \Phi_2, J, \Phi_3$ together with $v=g_+-g_-$ satisfy all of the desired properties.
\end{proof}

\begin{proof}[Proof of Theorem \ref{mainsa} in the general case]
Let $\Phi\colon \mathcal{A}_{sa}\to \mathcal{B}_{sa}$ be a linear isometry with finite corank. 
Then the bidual mapping $\Phi^{**}\colon \mathcal{A}^{**}_{sa} \to \mathcal{B}^{**}_{sa}$ is also a linear isometry with finite corank. 
Since $\mathcal{A}^{**}$, $\mathcal{B}^{**}$ are considered as von Neumann algebras, we see that the following holds: 
There exist central projections $p\in \mathcal{A}^{**}$, $q\in \mathcal{B}^{**}$ such that $p\mathcal{A}^{**}$, $q\mathcal{B}^{**}$ are finite-dimensional, and there exist
\begin{itemize}
\item a linear isometry $\Phi _1\colon p\mathcal{A}^{**}_{sa}\to q\mathcal{B}^{**}_{sa}$, 
\item a linear isometry $\Phi _2\colon (1-p)\mathcal{A}^{**}_{sa}\to (1-q)\mathcal{B}^{**}_{sa}$ that satisfies $\Phi _2(a)=vJ(a)$ for every $a\in (1-p)\mathcal{A}^{**}_{sa}$, where $J\colon (1-p)\mathcal{A}^{**}\to (1-q)\mathcal{B}^{**}$ is an injective unital Jordan $^*$-homomorphism with finite corank and $v=(1-q)\Phi^{**}(1)=(1-q)\Phi(1)$ is a central self-adjoint unitary in $(1-q)\mathcal{B}^{**}$, and
\item a linear mapping $\Phi _3\colon (1-p)\mathcal{A}^{**}_{sa}\to q\mathcal{B}^{**}_{sa}$, 
\end{itemize}
satisfying the following two conditions: 
\begin{itemize}
\item For every $a\in \mathcal{A}^{**}_{sa}$, we have $\Phi^{**} (a)=\Phi_1(pa)+\Phi_2((1-p)a)+\Phi_3((1-p)a)$, and 
\item the operator norm of the mapping $\mathcal{A}^{**}_{sa}\ni a\mapsto \Phi_1(pa)+\Phi_3((1-p)a)\in q\mathcal{B}^{**}_{sa}$ is at most $1$.
\end{itemize}
(One can complete the proof immediately in the case where $p\in \mathcal{A}$ and $q\in\mathcal{B}$ hold, but it is unclear whether this is the case. Thus we need a further discussion.)
Let us examine the two mappings 
\[
\Phi_4\colon \mathcal{A}_{sa}\ni a\mapsto \frac{\Phi(1)\Phi(a)+(\Phi(1)\Phi(a))^*}{2}\in \mathcal{B}_{sa}
\]
and  
\[
\Phi_5\colon \mathcal{A}_{sa}\ni a\mapsto \frac{\Phi(1)\Phi(a)-(\Phi(1)\Phi(a))^*}{2i}\in \mathcal{B}_{sa}.
\]
Then we have $\Phi(1)\Phi(a) = \Phi_4(a)+i\Phi_5(a)$ for every $a\in \mathcal{A}_{sa}$.
Set 
\[
S=\{\Phi_4(a^2)-\Phi_4(a)^2\mid a\in \mathcal{A}_{sa}\}\cup \Phi_5(\mathcal{A}_{sa})\cup\{\chi_{(-1,1)}(\Phi(1))\} \cup \{\Phi(1)b-b\Phi(1)\mid b\in \mathcal{B}\} \subset \mathcal{B}.
\]
Then it is readily seen that $S\subset q\mathcal{B}^{**}$. Since $q$ is a petty central projection of $\mathcal{B}^{**}$, we see that $S$ is petty as a subset of $\mathcal{B}$. Take the central projection $q_1\in \mathcal{B}$ such that the ideal generated by $S$ equals $q_1\mathcal{B}$. 
It follows that the mapping $\mathcal{A}_{sa}\ni a\mapsto (1-q_1)\Phi(1)\Phi(a)\in \mathcal{B}$ (whose image is actually contained in $\mathcal{B}_{sa}$) extends to a unital Jordan $^*$-homomorphism $J_1\colon \mathcal{A}\to (1-q_1)\mathcal{B}$. Moreover, $J_1$ has finite-dimensional kernel and finite corank. 
Now, take the central projection $p_1\in \mathcal{A}$ such that $\ker J_1=p_1\mathcal{A}$. 
The rest of the proof is almost the same as the case of von Neumann algebras. 
\end{proof}

\begin{example}
Let $J_1\colon \mathcal{A}\to \mathcal{B}$ be an injective unital Jordan $^*$-homomorphism with finite corank, and let $w\in \mathcal{B}$ be a self-adjoint unitary that commutes with every element of $J_1(\mathcal{A}_{sa})$, but assume that $w$ is not central in $\mathcal{B}$. Then $\Phi\colon \mathcal{A}_{sa}\ni a\mapsto wJ_1(a)\in \mathcal{B}_{sa}$ is a linear isometry with finite corank. 
For this $\Phi$, how can we take projections $p\in \mathcal{A}$, $q\in \mathcal{B}$ as in Theorem \ref{mainsa}? The answer is as follows. 

Take the projection $e\in \mathcal{B}$ satisfying $w=e-(1-e)$.
Since $J_1$ has finite corank, we may take operators $b_1, b_2, \ldots, b_m\in \mathcal{B}_{sa}$ which together with $J_1(\mathcal{A}_{sa})$ linearly span $\mathcal{B}_{sa}$. 
Since $e$ commutes with every element of $J_1(\mathcal{A}_{sa})$, we see that the space $\mathcal{V}=\{b\in \mathcal{B}_{sa}\mid ebe=(1-e)b(1-e)=0\}$ is at most $m$-dimensional. 
Therefore, Lemma \ref{eepetty} implies that $\mathcal{V}\subset \mathcal{B}$ is petty. One may take a petty central projection $q\in \mathcal{B}$ such that the ideal generated by $\mathcal{V}$ equals $q\mathcal{B}$. Then the mapping $J_2\colon \mathcal{A}\ni a\mapsto (1-q)J_1(a)\in (1-q)\mathcal{B}$ is a unital Jordan $^*$-homomorphism with finite-dimensional kernel. Thus one may take a petty central projection $p\in \mathcal{A}$ such that $\ker J_2=p\mathcal{A}$. 
One may check that $p$ and $q$ satisfy the desired property.
\end{example}

\subsection{Isometries with corank $1$}
Let us proceed with the case of corank $1$. We need several lemmas.

\begin{lemma}\label{l1}
Let $\Phi\colon \mathcal{A}_{sa}\to \mathcal{B}_{sa}$ be a linear isometry with corank $1$.
Set $f_+:=\chi_{\{1\}}(\Phi(1))$ and $f_-:=\chi_{\{-1\}}(\Phi(1))$. 
Then $f_+$ and $f_-$ are central projections in $\mathcal{B}$, and $1-f_+-f_-$ is either an atom in $\mathcal{B}$ or $0$. 
\end{lemma}
\begin{proof}
By Lemma \ref{e+-},  
the mapping 
\[
\Psi\colon \mathcal{A}_{sa}\ni a\mapsto f_+ \Phi(a)f_+-f_-\Phi(a)f_-\in f_+\mathcal{B}_{sa}f_+ +f_-\mathcal{B}_{sa}f_-
\]
is a unital isometry. 
Since $\Phi\colon \mathcal{A}_{sa}\to \mathcal{B}_{sa}$ is a linear isometry with corank $1$, we see that the real linear space $\{b\in \mathcal{B}_{sa}\mid f_+bf_+=f_-bf_-=0\}$ is at most $1$-dimensional. It follows that $f_+$, $f_-$ are central in $\mathcal{B}$ and $1-f_+-f_-$ is either $0$ or an atom in $\mathcal{B}$. 
\end{proof}

\begin{lemma}\label{l2}
Let $\Phi\colon \mathcal{A}_{sa}\to\mathcal{B}_{sa}$ be a linear isometry with corank $1$.
Assume that there is no central atom $q_0\in \mathcal{B}$ with the property that the mapping $\mathcal{A}_{sa}\ni a\mapsto (1-q_0)\Phi(a)\in (1-q_0)\mathcal{B}_{sa}$ is a surjective isometry.
Let $e\in \mathcal{A}$ be a central projection. Then there are a central projection $f\in \mathcal{B}$ and linear isometries $\hat{\Phi}\colon e\mathcal{A}_{sa}\to f\mathcal{B}_{sa}$, $\tilde{\Phi}\colon (1-e)\mathcal{A}_{sa}\to (1-f)\mathcal{B}_{sa}$ satisfying $\Phi(a)=\hat{\Phi}(ea)+\tilde{\Phi}((1-e)a)$ for every $a\in \mathcal{A}_{sa}$.
(Note that one of these mappings has corank $1$ and the other is surjective.)
\end{lemma}
\begin{proof}
Take $f_+, f_-$ as in the preceding lemma. Then the mapping 
\begin{equation}\label{Psi}
\Psi\colon \mathcal{A}_{sa}\ni a\mapsto f_+ \Phi(a)f_+-f_-\Phi(a)f_- =(f_+-f_-)\Phi(a)\in (f_++f_- )\mathcal{B}_{sa}
\end{equation}
is a unital isometry. 
The assumption implies that $f_++f_-=1$. (Indeed, if $1-f_+-f_-$ is a central atom, then $\Psi$ needs to be surjective because $\Phi$ is of corank $1$. Thus one may set $q_0=1-f_+-f_-$ to get a contradiction.)
Set $f_1:= \chi_{\{1\}}(\Psi(e))$ and $f_2:= \chi_{\{1\}}(\Psi(1-e))$. 
Then Lemma \ref{f1234} implies that the mapping $\mathcal{A}_{sa}\ni a\mapsto f_1\Psi(a)f_1+f_2\Psi(a)f_2\in \mathcal{B}_{sa}$ is a linear isometry. 
Since the corank of $\Psi$ is $1$, we see that the real linear space $\{b\in \mathcal{B}_{sa}\mid f_1bf_1=f_2bf_2=0\}$ is at most $1$-dimensional. Thus we see that $f_1$ and $f_2$ are central in $\mathcal{B}$, and moreover, using the assumption once again, we get $f_1+f_2=1$. %It then follows that $f_1$ is central in $\mathcal{B}$ as well. 
Therefore, setting $f=f_1$, $\hat{\Phi}(a)= f\Phi(a)$ for $a\in e\mathcal{A}_{sa}$, and  $\tilde{\Phi}(a)= (1-f)\Phi(a)$ for $a\in (1-e)\mathcal{A}_{sa}$, we see that the desired property holds. 
\end{proof}

\begin{lemma}\label{l3}
Let $\Phi\colon \mathcal{A}_{sa}\to \mathcal{B}_{sa}$ be a linear isometry with corank $1$, and $s\in \mathcal{A}$ a self-adjoint unitary. Then $\chi_{(-1,1)}(\Phi(s))$ is either a central atom in $\mathcal{B}$ or zero. 
\end{lemma}
\begin{proof}
Observe that $\Phi(s)\in \mathrm{ext}\, \Phi(\mathcal{A}_{sa})^{\leq 1}$ holds. Using Proposition \ref{ppp} and the fact that $\Phi(\mathcal{A}_{sa})\subset \mathcal{B}_{sa}$ is of codimension $1$, we see that $\chi_{(-1,1)}(\Phi(s))\in \mathcal{B}$ is either $0$ or an atom.
Assume that $\chi_{(-1,1)}(\Phi(s)) \mathcal{B} \chi_{\{1\}}(\Phi(s)) \neq \{0\}$. 
Then the dimension of this space as a real vector space is at least $2$. It follows that there is a nonzero element $b$ in this space such that $b+b^*\in \Phi(\mathcal{A}_{sa})$. 
Thus, by exactly the same discussion as in the proof of Lemma \ref{1-chi}, we get to a contradiction. Similarly, we get a contradiction by assuming $\chi_{(-1,1)}(\Phi(s)) \mathcal{B} \chi_{\{-1\}}(\Phi(s)) \neq \{0\}$. Thus we get $\chi_{(-1,1)}(\Phi(s)) \mathcal{B} (1-\chi_{(-1,1)}(\Phi(s)))=\{0\}$, which means that $\chi_{(-1,1)}(\Phi(s))$ is central.
\end{proof}

\begin{lemma}\label{l4}
Assume that $\mathcal{A}$ is finite-dimensional.
Let $\Phi\colon \mathcal{A}_{sa}\to \mathcal{B}_{sa}$ be a linear isometry with corank $1$.
Then there is a central atom $q_0\in \mathcal{B}$ such that the mapping $\mathcal{A}_{sa}\ni a\mapsto (1-q_0)\Phi(a)\in (1-q_0)\mathcal{B}_{sa}$ is a linear surjective isometry.
\end{lemma}
\begin{proof}
%If $n=1$, then $\mathcal{B}$ is a $2$-dimensional C$^*$-algebra, and it is clear that the desired conclusion holds. 
%In what follows we assume $n\geq 2$. 
We assume that there is no central atom $q_0\in \mathcal{B}$ with the property as in the statement. We aim to obtain a contradiction.
Take the central projections $f_{\pm}$ as in Lemma \ref{l1}. 
As in the proof of Lemma \ref{l2}, we see that $f_++f_-=1$.
%If $f_++f_-\neq 1$, then
%\[
%\Psi\colon \mathcal{A}_{sa}\ni a\mapsto f_+ \Phi(a)f_+-f_-\Phi(a)f_- =(f_+-f_-)\Phi(a)\in (f_++f_-)\mathcal{B}_{sa}
%\]
%needs to be surjective. 
%Thus, by setting $q_0=1-f_+-f_-$, we get to the desired conclusion.

%In what follows, we assume that $f_++f_-=1$. 
It suffices to study the unital linear isometry $\Psi$ as in \eqref{Psi} instead of $\Phi$. Moreover, by applying Lemma \ref{l2} repeatedly, we see that it suffices to consider the case where $\mathcal{A}$ equals $\mathbb{M}_n$ with $n\geq 1$. If $n=1$, then $\mathcal{B}$ is a $2$-dimensional C$^*$-algebra, and it is easy to show the existence of $q_0\in \mathcal{B}$ as in the statement. In what follows we assume $n\geq 2$. 

By imitating the proof of Lemma \ref{abe} in Case 1, and using Lemma \ref{l3}, we see that there is a central projection $f\in \mathcal{B}$ which is either $0$ or an atom such that the mapping $\mathcal{A}_{sa}\ni a\mapsto (1-f)\Psi(a)\in (1-f)\mathcal{B}_{sa}$ extends to a Jordan $^*$-homomorphism $J_1\colon \mathcal{A}\to (1-f)\mathcal{B}$. 
Since $\mathcal{A}=\mathbb{M}_n$ is simple and $\Psi$ is of corank $1$, we see that $J_1$ is injective. Now, an application of Theorem \ref{jhomo} shows that 
$\mathcal{B}$ contains an image of a nonzero $^*$-(anti)homomorphism from $\mathbb{M}_n$, which further implies that $\mathcal{B}$ contains $\mathbb{M}_n$ as a $^*$-subalgebra. Since $\mathcal{B}$ is $^*$-isomorphic to a direct sum of matrix algebras and $\dim \mathcal{B}=n^2+1$, we see that $\mathcal{B}$ may be identified with $\mathbb{M}_n\oplus \mathbb{C}$ as a C$^*$-algebra. 
Since $n\geq 2$, we see (in view of the general form of $^*$-(anti)homomorphisms from $\mathbb{M}_n$) that the above $f$ needs to correspond to $0\oplus 1\in \mathbb{M}_n\oplus \mathbb{C}$, and $J_1\colon \mathcal{A}\to (1-f)\mathcal{B}$ is a Jordan $^*$-isomorphism. Thus we again get a contradiction.
\end{proof}

\begin{proof}[Proof of Proposition \ref{corank1sa}]
By applying Theorem \ref{mainsa}, we get central projections $p\in \mathcal{A}$, $q\in \mathcal{B}$ and mappings $\Phi_1, \Phi_2, \Phi_3$ as in the statement. 
Since $\mathrm{corank}\, \Phi=1$, we get $(\mathrm{corank}\, \Phi_1, \mathrm{corank}\, \Phi_2)$ is either $(1,0)$ or $(0,1)$.  

If $(\mathrm{corank}\, \Phi_1, \mathrm{corank}\, \Phi_2)=(0,1)$, then we see that $\Phi_1\colon p\mathcal{A}_{sa}\to q\mathcal{B}_{sa}$ is a surjective isometry, which implies that $\Phi_1(p)$ is a central self-adjoint unitary in $q\mathcal{B}$. This further leads to $\Phi_3=0$, and we see that the mapping $\mathcal{A}_{sa}\ni a\mapsto \Phi(1)\Phi(a)\in \mathcal{B}$ extends to an injective unital Jordan $^*$-homomorphism from $\mathcal{A}$ to $\mathcal{B}$ with corank $1$.

Assume that $(\mathrm{corank}\, \Phi_1, \mathrm{corank}\, \Phi_2)=(1,0)$. Then $\Phi_1$ is a linear isometry with corank $1$ between the self-adjoint parts of finite-dimensional C$^*$-algebras. 
By Lemma \ref{l4}, there is a central atom $q_0\in q\mathcal{B}$ (which is also a central atom in $\mathcal{B}$) such that the mapping $e\mathcal{A}_{sa}\ni a\mapsto (q-q_0)\Phi_1(a)=(q-q_0)\Phi(a)\in (q-q_0)\mathcal{B}_{sa}$ is a surjective isometry.

From what we have shown, it is clear that $\lVert (1-q_0)\Phi(a)\rVert\geq \lVert a\rVert$, which together with the assumption that $\Phi$ is an isometry shows $\lVert (1-q_0)\Phi(a)\rVert= \lVert a\rVert$, for every $a\in \mathcal{A}_{sa}$. Since $q_0\in \mathcal{B}$ is a central atom and $\Phi$ has corank $1$, we see that the mapping $\mathcal{A}_{sa}\ni a\mapsto (1-q_0)\Phi(a)\in (1-q_0)\mathcal{B}_{sa}$ is a surjective isometry.
\end{proof}

\section{Complementary results}\label{other}
\subsection{Finite-codimensional Jordan $^*$-subalgebras}
\begin{lemma}\label{measure}
Let $X$ be a compact Hausdorff space and $\mathcal{A}_0\subset C(X)$ a finite-codimensional C$^*$-subalgebra. 
Let $\varphi\in C(X)^*$. Then we may also consider $\varphi$ as a complex measure on $X$. 
If $\mathcal{A}_0\subset \ker \varphi$ (where $\varphi$ is considered as a functional), then the support of $\varphi$ (considered as a complex measure) is a finite subset of $X$.
\end{lemma}
\begin{proof}
Since $\mathcal{A}_0\subset C(X)$ is finite-codimensional, we see that $X_0=\{x\in X\mid f(x)=0 \text{ for every }f\in \mathcal{A}_0\}$ is a finite set. 
Moreover, for each $x\in X$, the set $X_x=\{x_1\in X\mid f(x)=f(x_1)\text{ for every }f\in \mathcal{A}_0\}$ is a finite set. 
We also see that the set $X_1=\{x\in X\mid X_x \text{ has at least two points}\}$ is a finite set.
It follows by (the proof of) the Stone--Weierstrass theorem that $\mathcal{A}_0$ equals the algebra of all continuous functions on $X$ such that $f(X_0)=\{0\}$ and $f(X_x)$ is a singleton for each $x\in X$. 
In particular, the subalgebra $\{f\in C(X)\mid f(X_0\cup X_1)=\{0\} \}$ is contained in $\mathcal{A}_0$. 
If $\mathcal{A}_0\subset \ker \varphi$, then we have $\varphi(f)=0$ for every $f\in C(X)$ satisfying $f(X_0\cup X_1)=\{0\}$. 
Therefore, the support of $\varphi$ is contained in the finite set $X_0\cup X_1$.
\end{proof}

\begin{lemma}\label{bhsub}
Let $\mathcal{H}$ be an infinite-dimensional Hilbert space. 
Let $\varphi_1, \varphi_2, \ldots, \varphi_m$ be normal self-adjoint bounded linear functionals on $\mathbb{B}(\mathcal{H})$ whose support projections are of finite rank. If $\mathcal{A}_0= \bigcap_{j=1}^m\ker \varphi_j$ is a Jordan $^*$-subalgebra of $\mathbb{B}(\mathcal{H})$, then $\varphi_1= \varphi_2=\cdots= \varphi_m=0$ and $\mathcal{A}_0=\mathbb{B}(\mathcal{H})$.
\end{lemma}
\begin{proof}
Observe that there is a projection $p\in \mathbb{B}(\mathcal{H})$ with $\mathrm{rank}\, (1-p)<\infty$ that is orthogonal to the support projection of $\varphi_j$ for $j=1,2,\ldots, m$. Then  we have $p\mathbb{B}(\mathcal{H})p\subset \mathcal{A}_0$. 
Now, let $0\neq p\in \mathbb{B}(\mathcal{H})$ be an arbitrary projection satisfying $p\mathbb{B}(\mathcal{H})p\subset \mathcal{A}_0$ and $1\leq \mathrm{rank}\, (1-p)<\infty$. 
We can get to the desired conclusion by showing that there is another projection $p\neq q\geq p$ satisfying $q\mathbb{B}(\mathcal{H})q\subset \mathcal{A}_0$. 
Since $\mathcal{A}_0\subset \mathbb{B}(\mathcal{H})$ has finite codimension and $\dim p\mathbb{B}(\mathcal{H})(1-p)=\infty$, one may take a nonzero element $a\in \mathcal{A}_0$ satisfying $pa(1-p)=a$. Then one may take an atom $p_0\leq p$ such that $p_0a\neq 0$. 
It follows from $a, p_0\in \mathcal{A}_0$ that $\mathcal{A}_0\ni p_0a+ap_0=p_0a+pa(1-p)p_0=p_0a$. 
Since $p_0a=p_0a(1-p)$ is of rank $1$, one may take an atom $q_0\leq (1-p)$ satisfying $p_0a=p_0aq_0$. 
Then it is not hard to see that the Jordan $^*$-algebra generated by $p\mathbb{B}(\mathcal{H})p$ and $p_0a=p_0aq_0$ equals $q\mathbb{B}(\mathcal{H})q$, where $q=p+q_0$.
\end{proof}

\begin{proof}[Proof of Proposition \ref{J1}]
Since $\mathcal{A}_1\subset \mathcal{A}$ is a self-adjoint finite-codimensional closed subspace, one may find self-adjoint bounded linear functionals $\varphi_1, \varphi_2, \ldots, \varphi_m$ on $\mathcal{A}$ such that $\mathcal{A}_1=\bigcap_{j=1}^m \ker \varphi_j$. 
Observe that  each $\varphi_j$ can be considered as an element of the predual of the von Neumann algebra $\mathcal{A}^{**}$, and $\mathcal{A}_1^{**}$ can be identified with $\bigcap_{j=1}^m \ker \varphi_j$ in $\mathcal{A}^{**}$.
Then, $\mathcal{A}_1^{**}$ is in fact a Jordan $^*$-subalgebra of $\mathcal{A}^{**}$. (Indeed, for each $a\in \mathcal{A}_1^{**}$, we may take a net $a_\lambda\in \mathcal{A}_1$ converging to $a$ in the $\sigma$-weak ($=$ weak $^*$-) topology of $\mathcal{A}^{**}$. 
Thus we have $a_\lambda a_\mu+a_\mu a_\lambda\in \bigcap_{j=1}^m \ker \varphi_j$ for every $\lambda, \mu$. 
Taking the limit first with respect to $\mu$ and then with respect to $\lambda$ gives $2a^2\in \bigcap_{j=1}^m \ker \varphi_j$ and so $a^2\in \mathcal{A}_1^{**}$, as desired.)

We next show that the support projection of each $\varphi_j$ is petty in $\mathcal{A}^{**}$.
Take any maximal abelian von Neumann subalgebra (masa) $\mathcal{M}$ of $\mathcal{A}^{**}$. 
Let $\psi_j$ denote the restriction of $\varphi_j$ to $\mathcal{M}$.
Then $\mathcal{A}_1^{**}\cap \mathcal{M}$ is a $^*$-subalgebra of $\mathcal{M}$ with finite codimension. 
Therefore, by Lemma \ref{measure}, if we identify $\mathcal{M}$ with $C(X)$ for some compact Hausdorff space $X$, then the support of each $\psi_j$ considered as a complex measure on $X$ is a finite set. 
In other words, we see that the support projection of the self-adjoint normal linear functional $\psi_j\in \mathcal{M}_*$ is a sum of finitely many atoms in $\mathcal{M}$. 
Since we may take an arbitrary masa $\mathcal{M}\subset \mathcal{A}^{**}$ in the above discussion, we see that the support projection of each $\varphi_j$ is a sum of finitely many atoms in $\mathcal{A}^{**}$.
(Indeed, one may take a masa in $\mathcal{A}^{**}$ which contains both the support projections of the negative and positive parts of $\varphi_j$. Then the above discussion shows that these support projections are sums of atoms in the masa. Since an atom in a masa is also an atom in $\mathcal{A}^{**}$, we get the desired fact.)
Let $p$ be a central projection of $\mathcal{A}^{**}$ such that $p\mathcal{A}^{**}$ is $^*$-isomorphic to $\mathbb{B}(\mathcal{H})$ for some infinite-dimensional Hilbert space $\mathcal{H}$. 
Then the space $p\mathcal{A}^{**}\cap \mathcal{A}_1^{**}\subset p\mathcal{A}^{**}$ is a Jordan $^*$-subalgebra with finite codimension. It follows from Lemma \ref{bhsub} that $p\mathcal{A}^{**}\cap \mathcal{A}_1^{**}= p\mathcal{A}^{**}$. Therefore, we see that the support projection of each $\varphi_j$ is orthogonal to $p$. 
Consequently, the support projection of each $\varphi_j$ is petty in $\mathcal{A}^{**}$. 

Take the central projection $e\in \mathcal{A}^{**}$ such that the ideal generated by these support projections equals $\mathcal{A}_2=e\mathcal{A}^{**}$, which is a finite-dimensional C$^*$-algebra. Then the mapping $\pi\colon \mathcal{A}\ni a\mapsto ea\in \mathcal{A}_2$ is a $^*$-homomorphism. 
Moreover, we see that $\mathcal{A}_3= \mathcal{A}_2\cap \mathcal{A}_1^{**}$ is a Jordan $^*$-subalgebra of $\mathcal{A}_2$ satisfying $\mathcal{A}_1=\{a\in \mathcal{A}\mid \pi(a)\in \mathcal{A}_3\}$. 
Because $\mathcal{A}\subset \mathcal{A}^{**}$ is weak $^*$-dense and $\mathcal{A}_2$ is finite-dimensional, we see that $\pi$ is surjective.
Thus the equality $\mathcal{A}_1=\{a\in \mathcal{A}\mid \pi(a)\in \mathcal{A}_3\}$ leads to $m=\dim \mathcal{A}_2-\dim \mathcal{A}_3$.
\end{proof}

\begin{lemma}\label{kl}
Let $\mathcal{A}=\bigoplus_{j=1}^m \mathbb{M}_{n_j}$ be a finite-dimensional C$^*$-algebra. Let $\mathcal{A}_0\subset \mathcal{A}$ be a Jordan $^*$-subalgebra with codimension $1$. Then one of the following holds.
\begin{itemize}
\item There is $1\leq k\leq m$ satisfying $n_{k}=1$ and 
\[
\mathcal{A}_0=\bigoplus_{1\leq j\leq m,\, j\neq k} \mathbb{M}_{n_j} \subset \mathcal{A}.
\] 
\item There are $1\leq k<l\leq m$ satisfying $n_{k}=n_{l}=1$ and 
\[
\mathcal{A}_0=\{a=(A_j)_{1\leq j\leq m}\in \bigoplus_{j=1}^m \mathbb{M}_{n_j}\mid A_{k}=A_{l}\}\subset \mathcal{A}.
\] 
\item There is $1\leq k\leq m$ satisfying $n_{k}=2$ and a unitary matrix $U\in \mathbb{M}_2$ such that 
\[
\mathcal{A}_0=\{a=(A_j)_{1\leq j\leq m}\in \bigoplus_{j=1}^m \mathbb{M}_{n_j}\mid (UA_{k}U^*)_{11}=(UA_{k}U^*)_{22}\}\subset \mathcal{A}.
\]
\end{itemize}
\end{lemma}
\begin{proof}
Since $\mathcal{A}_0\subset \mathcal{A}$ is self-adjoint and of codimension $1$, one may find a nonzero self-adjoint linear functional $\varphi$ on $\mathcal{A}$ satisfying $\mathcal{A}_0=\ker \varphi$. Then there are hermitian matrices $B_j\in \mathbb{M}_{n_j}$, $1\leq j\leq m$, such that $\varphi(a)=\sum_{j=1}^m\mathrm{tr} (A_jB_j)$ for every $a=(A_j)_{1\leq j\leq m}\in \bigoplus_{j=1}^m \mathbb{M}_{n_j}=\mathcal{A}$. 
By diagonalizing each $B_j$ by a unitary matrix, we see that it suffices to consider the case where each $B_j$ is a diagonal matrix. 
Moreover, by looking at the commutative C$^*$-algebra $\mathcal{D}$ formed of the direct sum of diagonal matrices in $\mathcal{A}=\bigoplus_{j=1}^m \mathbb{M}_{n_j}$, and its $^*$-subalgebra $\mathcal{D}\cap \mathcal{A}_0$ with codimension $1$, we also see that the support projection of $\varphi$ is the sum of at most $2$ atoms in $\mathcal{A}$.
(To see this, imitate the proof of Lemma \ref{measure}.) 
Now, it is easily seen by some elementary calculations that $\ker \varphi\subset \mathcal{A}$ forms a Jordan $^*$-subalgebra if and only if one of the following three conditions holds.
\begin{itemize}
\item There is $1\leq k\leq m$ satisfying $n_{k}=1$, $B_{k}\neq 0$, and $B_j=0$, $j\in \{1,2,\ldots, m\}\setminus \{k\}$.  
\item There are $1\leq k<l\leq m$ satisfying $n_{k}=n_{l}=1$, $B_{k}=-B_{l}\neq 0$, and $B_j=0$, $j\in \{1,2,\ldots, m\}\setminus \{k, l\}$.
\item There is $1\leq k\leq m$ satisfying $n_{k}=2$, $(B_{k})_{11}=-(B_{k})_{22}\neq 0$, and $B_j=0$, $j\in \{1,2,\ldots, m\}\setminus \{k\}$. 
\end{itemize}
Thus we get the desired conclusion.
\end{proof}

\begin{proof}[Proof of Proposition \ref{J2}] 
By Proposition \ref{J1}, there is a surjective $^*$-homomorphism $\pi\colon \mathcal{A}\to \mathcal{A}_2$ onto a finite-dimensional C$^*$-algebra $\mathcal{A}_2$ and a Jordan $^*$-subalgebra $\mathcal{A}_3\subset \mathcal{A}_2$ with codimension $1$ such that $\mathcal{A}_1=\{a\in \mathcal{A}\mid \pi(a)\in \mathcal{A}_3\}$.
Then Lemma \ref{kl} gives the explicit form of $\mathcal{A}_3$. 
By considering the restriction of $\pi$ to a suitable direct summand of $\mathcal{A}_2$, we get the desired conclusion. 
\end{proof}

\subsection{Order embeddings}
\begin{lemma}
Let $\Phi\colon \mathcal{A}_{sa}\to\mathcal{B}_{sa}$ be a positive linear mapping with finite corank. 
Then either $\Phi(1)$ is invertible or $0$ is an isolated point of $\sigma(\Phi(1))$, and $\chi_{\{0\}}(\Phi(1))\in \mathcal{B}$ is petty.
\end{lemma}
\begin{proof}
Assume that $0\in \sigma(\Phi(1))$ and it is not an isolated point of $\sigma(\Phi(1))$.
Let $F$ denote the real vector space of continuous functions on $[0,\infty)$ spanned by the functions $[0,\infty)\ni t\mapsto t^{1/m}\in \mathbb{R}$, $m\geq 2$. 
Then the linear space $\{f(\Phi(1))\mid f\in F\}$ is infinite-dimensional. 

Let $a_2, a_3, \ldots, a_n\in \mathbb{R}$ and assume that $a_n\neq 0$. 
Set $g(t)=a_2t^{1/2}+a_3 t^{1/3}+\cdots +a_nt^{1/n}$. 
Then we have $g(\Phi(1))=a_2\Phi(1)^{1/2}+a_3 \Phi(1)^{1/3}+\cdots +a_n\Phi(1)^{1/n}$. 
Considering $\Phi(1)$ as an element of a commutative C$^*$-algebra, and using the assumption that $0$ is not an isolated point of $\sigma(\Phi(1))$, we see that there is no $C>0$ satisfying $-C\Phi(1)\leq g(\Phi(1))\leq C\Phi(1)$. 

On the other hand, if $a\in \mathcal{A}_{sa}$, then we have $-\lVert a\rVert\Phi(1)\leq \Phi(a)\leq \lVert a\rVert\Phi(1)$.
It follows that $\{f(\Phi(1))\mid f\in F\}\cap \Phi(\mathcal{A}_{sa})=\{0\}$. 
This contradicts the assumption that $\Phi$ has finite corank. 
Therefore, we see that either $\Phi(1)$ is invertible or $0$ is an isolated point of $\sigma(\Phi(1))$. 
Moreover, $\mathcal{V}=\{b\in \mathcal{B}_{sa}\mid \chi_{(0,\infty)}(\Phi(1))b\chi_{(0,\infty)}(\Phi(1))=0\}$ satisfies $\mathcal{V}\cap \Phi(\mathcal{A}_{sa})=\{0\}$. Thus $\mathcal{V}$ is finite-dimensional, and Lemma \ref{eepetty} implies that $\chi_{\{0\}}(\Phi(1)) \in\mathcal{B}$ is petty.
\end{proof}

\begin{proof}[Proof of Theorem \ref{embedding}]
Let $\Phi\colon \mathcal{A}_{sa}\to \mathcal{B}_{sa}$ be a linear order embedding of finite corank. 
Set $f=\chi_{\{0\}}(\Phi(1))$. 
By the preceding lemma, we see that $f\in \mathcal{B}$ is petty, and $(1-f)\Phi(1)(1-f)\in (1-f)\mathcal{B}_{sa}(1-f)$ is positive and invertible. Take the square root $b$ of this operator.
Since the range of the operator $\Phi(a)$ is contained in the range of $1-f$ for each $a\in \mathcal{A}_{sa}$, we may define the mapping $\Psi\colon \mathcal{A}_{sa}\to (1-f)\mathcal{B}_{sa}(1-f)$ by $\Psi(a)=b^{-1} \Phi(a)b^{-1}$. 
Then $\Psi$ is a unital linear order embedding with finite corank, which is a linear isometry by Lemma \ref{psi}, and we have $\Phi(a)=b\Psi(a)b$ for every $a\in \mathcal{A}_{sa}$.
The converse is also a consequence of Lemma \ref{psi}.
\end{proof}

\section{Questions}
We still do not know whether every unital positive linear isometry $\Phi\colon \mathscr{H}_n\to\mathscr{H}_k$ is a generalized CLP isometry.
A remarkable fact from the proof of Lemma \ref{MA} is that a generalized CLP isometry $\Phi \colon \mathscr{H}_n\to \mathscr{H}_k$ always satisfies $\sigma(X)\subset \sigma(\Phi(X))$ for every $X\in \mathscr{H}_n$. 
It should be of interest to ask the following question. 
\begin{question}
Is there a unital positive linear isometry $\Phi \colon \mathscr{H}_n\to \mathscr{H}_k$ satisfying $\sigma(X)\not\subset \sigma(\Phi(X))$ for some $X\in \mathscr{H}_n$?
\end{question}
By Lemma \ref{psi}, there is no such mapping if $n=2$. 
Interestingly, mappings $\Phi$ satisfying the ``dual'' condition $\sigma(X)\supset \sigma(\Phi(X))$ for every $X$ is recently studied (but in a slightly different context) for example in \cite{CGT} (see also references therein).

In the commutative case, the study of isometric shifts is an area with fruitful research activity. 
We gave in Proposition \ref{corank1} a generalization of the result by Gutek, Hart, Jamison, and Rajagopalan \cite{GHJR}. That proposition shows that there are two types of isometric shifts even in the setting of noncommutative C$^*$-algebras. 
A naive question is the following. 
\begin{question}
Is it possible to consider variants of the results for commutative C$^*$-algebras given for example in \cite{FV}, \cite{H}, \cite{AF2}, \cite{GN}, \cite{A}, and \cite{A2} in the setting of noncommutative C$^*$-algebras?
\end{question}

Lastly, we pose the following question.
\begin{question}
The original theorems by Kadison in \cite{K}, \cite{K2} are generalized in various directions. In particular, many results on linear surjective isometries between certain classes of real or complex Banach spaces with certain algebraic structures are known. See for example \cite{FMP} and references therein.
Is it possible to think of generalizations of our theorems to such directions?
\end{question}

\medskip\medskip

\noindent 
\textbf{Acknowledgements.} \quad 
This work is originally motivated by a question by Hiroyasu Hamada (Okayama University of Science) to the author which asks whether there is an isometric shift on $\mathbb{B}(\mathcal{H})$, to which the author could not answer immediately. 
By Corollary \ref{hamada}, the author now can tell him that the answer is negative.
The author is grateful to him for asking this question and also for allowing the author to visit Sasebo, where intensive discussions on related topics are made.

Also, the author appreciates the anonymous referee for careful reading of the paper with a detailed report, which substantially improved the quality of the paper.

\end{document}